\documentclass{article}
\usepackage{amssymb}
\usepackage{amsmath}
\usepackage{amsthm, amsfonts, mathrsfs}

\usepackage{float}

\usepackage{cmll}

\usepackage{tikz}

\usetikzlibrary{calc}

\newtheorem{thm}{Theorem}[section]

\newtheorem{prop}[thm]{Proposition}

\newtheorem{lem}[thm]{Lemma}

\newtheorem{qu}[thm]{Question}
\newtheorem{ex}[thm]{Example}

\newtheorem{conj}{Conjecture}

\newcommand{\comment}[1]{}

\newcommand{\tq}{\ge_T}
\newcommand{\te}{=_T}

\newcommand{\gtq}{\ge_{T}}
\newcommand{\gte}{=_{T}}

\newcommand{\K}{\mathcal{K}}

\newcommand{\omom}{\omega^\omega}

\newcommand{\cl}[1]{\overline{#1}}

\DeclareMathOperator{\ci}{c}
\DeclareMathOperator{\nci}{\neg c}

\DeclareMathOperator{\cof}{cof}
\DeclareMathOperator{\add}{add}

\title{The Shape of  Generating Families}

\author{Ziqin Feng and Paul Gartside}
\date{}

\begin{document}

\maketitle

\begin{abstract}
The topology of a space $X$ is \emph{generated} by a family $\mathcal{C}$ of its subsets provided that a set $A\subseteq X$ is closed in $X$ if and only if $A\cap C$ is closed in $C$ for each $C\in \mathcal{C}$.
A space $X$ is a \emph{$k$-space} (respectively, \emph{sequential})  if its topology is generated by the collection of all compact subsets (respectively, convergent sequences) of $X$.

Relations are defined to capture the notion of a space being a $k$-space or sequential.
The structure (or `shape') under the Tukey order of these relations applied to separable metrizable spaces  is examined.  For the $k$-space case the initial structure is completely determined, and the cofinal structure is shown to be highly complex. In the sequential case, however, the entire shape is determined. It follows  that the number of Tukey types in the sequential case lies between $\aleph_0$ and $\mathfrak{c}$, is equal to $\aleph_0$ precisely when $\mathfrak{c} < \aleph_{\omega_1}$, and is equal to $\mathfrak{c}$ if and only if $\mathfrak{c}$ is a fixed point of the aleph function, necessarily of  uncountable cofinality.


\end{abstract}

\section{Introduction}
In his survey article, \cite{vD:hbk}, on cardinal characteristics of the continuum and small cardinals arising in topology, van Douwen introduced three such invariants of a separable metrizable space, $M$, namely $\cof(\K(M))$, $\mathop{kc}(M)$ and $\mathop{k}(M)$. Each invariant asks for the minimum size of a family of compact subsets of $M$ with certain properties.
The third invariant, $\mathop{k}(M)$, requires that the compact subsets witness the $k$-space property of $M$. The purpose of the present paper is to understand  not just the size, but the `shape', $\mathbf{k}(M)$, of compact families witnessing the $k$-space property ($k$-structures), and the `shape', $\mathbf{seq}(M)$, of families of convergent sequences witnessing sequentiality (sequential structures), of a separable metrizable space.
We uncover the initial organization
of $k$-structures (for small or simple $M$), but we show, in ZFC, that there is a huge number, $2^\mathfrak{c}$,  of $k$-structures, organized in a highly complex way that we can have no hope of understanding in detail.
In striking contrast, we establish, in ZFC and in almost complete detail, the organization of sequential structures of separable metrizable spaces. A notable consequence is that there are continuum, $\mathfrak{c}$, many distinct sequential structures (for separable metrizable spaces) if and only if $\mathfrak{c}$ is a fixed point of the aleph-function. The smallest possible value this could occur is the $\omega_1$st-fixed point of the aleph function, and this is consistent.
The second author and Thomas Gilton have applied these results, see \cite{GG-CofGen}, to compute $k(M)$ and $\mathop{seq}(M)$, the cofinality of $\mathbf{seq}(M)$, for many separable metrizable spaces, $M$, revealing a connection to Shelah's \textsc{PCF} theory.

To motivate our results we review the small cardinals, $\mathfrak{b}$ and $\mathfrak{d}$, and examine the first two of van Douwen's invariants.
As usual, by $\omega$ we mean the initial infinite ordinal (so, depending on context, it is interpreted as $\aleph_0$, the first infinite cardinal, or as a  well-ordered set), and $\omega_1$ the initial uncountable ordinal. Then $\omom=(\omom,\le)$ is all functions of $\omega$ to itself with the product order, and $(\omom,\le^*)$ is the same set but with  the mod-finite order.
The \emph{bounding number}, $\mathfrak{b}$, is the minimum size of an unbounded set in $(\omom,\le^*)$, while the \emph{dominating number}, $\mathfrak{d}$, is the minimum size of a cofinal set in $(\omom,\le^*)$, or equivalently, in $(\omom,\le)$. Note $\omega_1 \le \mathfrak{b} \le \mathfrak{d} \le \mathfrak{c}$.

The first of van Douwen's invariants, $\cof(\K(M))$, is the cofinality of the directed set $\K(M)$ of all compact subsets of $M$, ordered by inclusion, $\subseteq$. Van Douwen observed that $\cof(\K(M))=1$ precisely when $M$ is compact, while $\cof(\K(M)) \le \omega$ if and only if $M$ is locally compact, but if $M$ is not locally compact then $\mathfrak{d} \le \cof(\K(M)) \le \mathfrak{c}$, and he computed $\cof(\K(M))$ for certain specific separable metrizable $M$.
In fact the underlying reason for these results had already been unearthed by Fremlin \cite{Fr2}. If $P$ and $Q$ are two directed sets, then $P$ Tukey quotients to $Q$, denoted $P \tq Q$, if there is a map $\phi:P \to Q$ carrying cofinal sets of $P$ to cofinal sets of $Q$. Evidently, if $P \tq Q$ then $\cof(P) \ge \cof(Q)$, and the relations that van Douwen mentioned between various $\cof(\K(M))$ were all due to a corresponding Tukey quotient. See \cite{GM1, GM2} for further results on  $\K(M)$'s and the Tukey relation.

The second invariant, $\mathop{kc}(M)$, is the minimal size of a compact cover of the separable metrizable space, $M$.
Van Douwen noted that, always, $ \cof(\K(M)) \ge \mathop{kc}(M)$, while $\mathop{kc}(M) \le \omega$ precisely when $M$ is $\sigma$-compact, and again he  computed $\mathop{kc}(M)$ for various $M$.
This invariant can also be interpreted as the cofinality of a suitable object,  but not of a directed set. Further, the connections and computations are due to morphisms between the relevant objects.
The objects are relations.
By a relation we mean a triple $\mathbf{A}=(A_-,A_+,A)$, where $A_-$ is the domain, $A_+$ the range, and $A$ is a collection of pairs.
We write, as usual, $xAy$ for $(x,y) \in A$.
Every directed set is a relation (where the domain and range happen to coincide), and we borrow concepts about directed sets appropriately. In particular, a subset $C$ of $A_+$ is \emph{cofinal} in $\mathbf{A}$ if for every $x$ from $A_-$ there is a $y$ from $C$ such that $xAy$, and  the cofinality of $\mathbf{A}$, denoted $\cof(\mathbf{A})$, is the minimal size of a cofinal set.
Letting $\mathbf{kc}(M)$ be the relation $(M,\K(M),\in)$, we see
 $\mathop{kc}(M)=\cof(\mathbf{kc}(M))$.
The idea that many cardinal characteristics of the continuum could be viewed as the cofinality of a suitable relation, and that the connections between the small cardinals were due to suitable morphisms was implicit in the work of Miller, Fremlin and others, but was made explicit by Vojtas \cite{Voj}. In particular, the well known Cichon's diagram follows from a corresponding diagram of morphisms of relations, see \cite{FremCich}.
This was clarified by Blass, who made extensive use of relations and morphisms in his survey on small cardinals \cite{BlassHB}.
In Blass' notation (simplified by Lemma~\ref{l:ulm} below) a Tukey morphism from $\mathbf{A}$ to $\mathbf{B}$ is a map $\phi_+:A_+ \to B_+$ carrying cofinal sets to cofinal sets. If such a morphism exists we write $\mathbf{A} \gtq \mathbf{B}$.
The relation $\gtq$ is transitive, hence $\gte$, defined by $\mathbf{A} \gte \mathbf{B}$ if and only if $\mathbf{A} \gtq \mathbf{B}$ and $\mathbf{B} \gtq \mathbf{A}$, is an equivalence relation.
The organization of $\mathbf{kc}(M)$'s under $\gtq$ is explored in \cite{FG-ShapeCptCovers,GM1}.

Now we turn back to van Douwen's invariant, $\mathop{k}(M)$, connected to the $k$-space property.
The topology of a space $X$ is said to be \emph{generated} by a family $\mathcal{C}$ of its subsets provided that a set $A\subseteq X$ is closed in $X$ if and only if $A\cap C$ is closed in $C$ for each $C\in \mathcal{C}$.
Observe this is equivalent to saying if $F$ is a subset of $X$  which is not closed  then there is a $C$ in $\mathcal{C}$ such that $F \cap C$ is not closed in $C$.
A space $X$ is a \emph{$k$-space}, or \emph{compactly generated},  if its topology is generated by the collection of all compact subsets of $X$.
While $X$ is a \emph{sequential} space, or \emph{sequential},  if its topology is generated by the collection of all (non-trivial) convergent sequences.
We define  relations to capture the notion of a space being compactly generated or sequential.
Let $X$ be a space, and let $\ci$ be defined for subsets $S$ and $T$ of $X$ by, $S \ci T$ if and only if $S \cap T$ is  closed in $T$, so
$S \nci T$ if and only if $S \cap T$ is not closed in $T$.
Set $\mathop{NC}(X)$ to be all subsets of $X$ that are not closed, and let
 $\mathop{CS^+}(X)$ be all infinite convergent sequences with limit.
Now we see that $X$ is a $k$-space if and only if $\mathbf{k}(X)=(\mathop{NC}(X),\K(X),\nci)$ has a cofinal set, and more generally, that cofinal sets correspond to generating compact collections.
We call $\mathbf{k}(X)$ the \emph{$k$-structure} of $X$, and note $\mathop{k}(X)=\cof(\mathbf{k}(X))$.
Similarly, $X$ is sequential if and only if $\mathbf{seq}(X)=(NC(X),\mathop{CS}^+(X),\nci)$ has a cofinal set, and more generally cofinal sets correspond to generating collections of (infinite) convergent sequences.
We call $\mathbf{seq}(X)$ the \emph{sequential structure} of $X$.

By $\mathbf{k}(\mathcal{M})$ we mean the  Tukey types ($\gte$ equivalence classes) of $\mathbf{k}(M)$, where $M$ is separable metrizable, ordered by $\gtq$.
Our results on the shape of $(\mathbf{k}(\mathcal{M}),\gtq)$, obtained in Section~\ref{s:k}, are summarized in the following diagram. Here types increase in the  Tukey order from left to right (simplest relations to the left, most complex to the right). A black line between two types indicates both the order, and the fact that there is no type strictly between them.
In the diagram $M$ denotes a separable metrizable space, $M'$ is the set of non-isolated points, $M^\#$ the set of points of non-local compactness. Recall $\omega$ is well-ordered, for the special relation $(\omom,\le_\infty)$ see Section~\ref{ss:ssr}, $\omom$ has the product order, and for the operation on relations, $\&$, see Section~\ref{ss:sops}.
The diagram highlights our complete understanding of the small types of $\mathbf{k}(M)$, but also the complexity and size of larger types (the shaded areas, containing order-isomorphic copies of the $2^\mathfrak{c}$-antichain and the unit interval under the usual order, for example).
An area of particular ambiguity are $\mathbf{k}(M)$ strictly above $(\omom,\le_\infty) \times \omega$ but incomparable with $\omom$. We know such $M$ must be Menger but not $\sigma$-compact. We know it is consistent there are many such $M$. But we are far from finding ZFC examples, if they exist.

\begin{figure}[H]
\caption{$(\mathbf{k}(\mathcal{M}),\le_{T})$, for separable metrizable $M$}
\noindent\begin{tikzpicture}
\shadedraw[left color=cyan, right color=white,draw=white]  (6.5,0) to [out=80,in=180]  (11.5,3.0) to  (11.5,-3.0) to [out=180,in=-80]  cycle;

\shadedraw[left color=blue!40, right color=white,draw=white]  (8.5,0) to [out=90,in=180]  (11.5,2.5) to  (11.5,-2.5) to [out=180,in=-90]  cycle;

\draw (0,0) -- (1.5,0);
\node[fill=red,circle,inner sep=1.5] (z) at (0,0) [label=above:$\mathbf{0}$,label=below:{\footnotesize discrete}] {};

\draw (1.5,0) -- (3,2) -- (4.25,0) -- (3,-2) -- cycle;
\draw (4.25,0) -- (8.5,0);
\node[fill=red,circle,inner sep=1.5] (one) at (1.5,0) [label=above:$\mathbf{1}$,label={[text width=1cm,execute at begin node=\setlength{\baselineskip}{1ex},fill=white,inner sep=1.5] below:{\footnotesize  compact $\oplus$\,discrete}}] {};


\node[fill=red,circle,inner sep=1.5] (F) at (3,2) [label=right:{$(\omom,\le_\infty)$},label=above left:{\footnotesize $M'$ compact, $M^\#\ne\emptyset$}] {};

\node[fill=red,circle,inner sep=1.5] (lc) at (3,-2) [label=right:{$\omega$},label=below:{\footnotesize $M'$ not compact, $M^\#=\emptyset$}] {};

\node[fill=red,circle,inner sep=1.5] (Msh) at (4.25,0) [label={[fill=white,fill opacity=1.0, text opacity =1, inner sep=2] above:{$(\omom,\le_\infty) \with \omega$}},label={[fill=white,fill opacity=1.0, text opacity =1, inner sep=2] below:{\footnotesize $M^\#$ compact}}] {};

\node[fill=red,circle,inner sep=1.5] (SK) at (6.5,0) [label={[fill=white,fill opacity=0.75, text opacity =1] above:{$(\omom,\le_\infty) \times \omega$}},label={[fill=white,fill opacity=0.75, text opacity =1,text width=2.325cm,execute at begin node=\setlength{\baselineskip}{1.5ex}] below:{\footnotesize $M'$ $\sigma$-compact  $M^\#$ not compact}}] {};


\node[fill=red,circle,inner sep=1.5] (A) at (8.5,0) [label={[fill=white,fill opacity=0.75, text opacity =1] above:{$\omom$}},label={[fill=white,fill opacity=0.75, text opacity =1,text width=1.75cm,inner sep=1,execute at begin node=\setlength{\baselineskip}{1.5ex}] below right:{\footnotesize analytic not $\sigma$-compact}}] {};

\node[fill=white,fill opacity=0.75, text opacity =1,text width=2.03cm,inner sep=1,execute at begin node=\setlength{\baselineskip}{1.5ex}] (Men) at (6.5,-2.25) {\footnotesize (consistently) $\exists$ many  Menger not $\sigma$-compact};

\draw[->,thin] (Men.east) to[out=20,in=200] (7.85,-1.25);

\node[fill=white,fill opacity=0.75, text opacity =1,text width=2.05cm,inner sep=1,execute at begin node=\setlength{\baselineskip}{1.5ex}] (NotMen) at (10.15,-1.35) {\footnotesize all not Menger};

\draw[red,dashed] (11.2,2)--(11.2,-2.0);

\node[rounded corners,fill=red!10,fill opacity=0.75, text opacity =1,text width=1.6cm,inner sep=1,execute at begin node=\setlength{\baselineskip}{1.5ex}] (NotMen) at (10.25,1.75) {\footnotesize $2^\mathfrak{c}$ antichain};

\node[rounded corners, fill=red!10,fill opacity=0.75, text opacity =1,text width=1.7cm,inner sep=1,execute at begin node=\setlength{\baselineskip}{1.5ex}] (NotMen) at (10.1,0.9) {\footnotesize copies of $\mathfrak{c}^+$, $(\mathbb{P}(\omega),\subseteq)$, $I$};

\end{tikzpicture}
\end{figure}
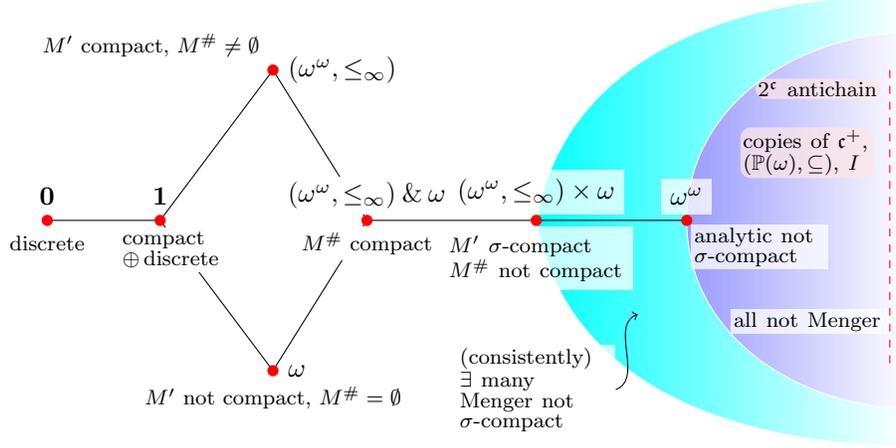

Similarly  $\mathbf{seq}(\mathcal{M})$ denotes the  Tukey types of $\mathbf{seq}(M)$ for $M$ separable metrizable, ordered by $\gtq$.
But in this case we do not have shaded regions of high, and poorly understood, complexity.
Instead, as established in Section~\ref{s:seq} the types are ranked by the cardinality, $\aleph_\alpha$, of the relevant separable metrizable space, with only those of cardinality with countable cofinality contributing more than a single type.

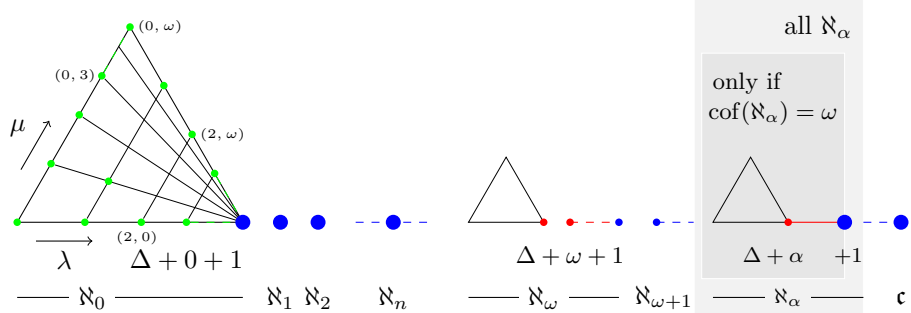
\begin{figure}[H]
\caption{$(\mathbf{seq}(\mathcal{M}),\le_{T})$, for $M$ with $|M|=\aleph_\alpha$}
\noindent\begin{tikzpicture}

\coordinate (Or) at (0,0);
\coordinate (Top) at (60:3cm);
\coordinate (Rt) at (0:3cm);

\draw[thin] (Or) -- (Top) -- (Rt) -- cycle;

\node[fill=green,inner sep=1,circle]  at (Or) {};

\node[fill=green,inner sep=1,circle]  at (Top) {};

\node[fill=green,inner sep=1,circle] at (barycentric cs:Or=0.7,Top=0.0,Rt=0.3) (OZ) {};

\node[fill=green,inner sep=1,circle] at (barycentric cs:Or=0.45,Top=0.0,Rt=0.55) (TwoZ) {};

\node[fill=green,inner sep=1,circle] at (barycentric cs:Or=0.25,Top=0.0,Rt=0.75) (ThreeZ) {};

\node at ($(TwoZ.west)+(0,-0.2)$) {\tiny $(2,0)$};

\node[fill=green,inner sep=1,circle] at (barycentric cs:Or=0.0,Top=0.7,Rt=0.3) (OOm) {};

\node at ($(Top.north)+(0.4,0)$) {\tiny $(0,\omega)$};

\node[fill=green,inner sep=1,circle] at (barycentric cs:Or=0.0,Top=0.45,Rt=0.55) (TwoOm) {};

\node[fill=green,inner sep=1,circle] at (barycentric cs:Or=0.0,Top=0.25,Rt=0.75) (ThreeOm) {};

\node at ($(TwoOm.east)+(0.35,0)$$) {\tiny $(2,\omega)$};

\node[fill=green,inner sep=1,circle] at (barycentric cs:Or=0.7,Top=0.3,Rt=0.0) (ZO) {};

\node[fill=green,inner sep=1,circle] at (barycentric cs:Or=0.45,Top=0.55,Rt=0.0) (ZTwo) {};

\node[fill=green,inner sep=1,circle] at (barycentric cs:Or=0.25,Top=0.75,Rt=0.0) (ZThree) {};

\node at ($(ZThree.west)+(-0.3,0)$) {\tiny $(0,3)$};

\draw[very thin] (barycentric cs:Or=0.1,Top=0.9,Rt=0.0) -- (Rt);

\draw[very thin] (ZO) -- (Rt);
\draw[very thin] (ZTwo) -- (Rt);
\draw[very thin] (ZThree) -- (Rt);

\draw[very thin] (OZ) -- (OOm);
\draw[very thin] (TwoZ) -- (TwoOm);
\draw[very thin] (ThreeZ) -- (ThreeOm);

\draw[dashed,green] (ThreeZ) -- (Rt);
\draw[dashed,green] (ThreeOm) -- (Rt);
\draw[dashed,green] (ZThree) -- (Top);

\node[fill=green,inner sep=1,circle] at ($(OZ)!0.3!(OOm)$)  {};

\node[fill=blue,inner sep=2,circle] at (Rt) {};

\draw[->] (0.25,-0.25) -- (1,-0.25) node[midway,below] {$\lambda$};

\draw[->] (0.05,+0.7) -- ++(60:0.75) node[midway,above left] {$\mu$};

\draw[thin] (0,-1) -- (3,-1);
\node[fill=white] at (1,-1) {$\aleph_0$};

\node at (2.25,-0.5) {$\Delta+0+
1$};

\node[fill=white] at (3.5,-1) {$\aleph_1$};
\node[fill=blue,inner sep=2,circle] at (3.5,0) {};
\node[fill=white] at (4.0,-1) {$\aleph_2$};
\node[fill=blue,inner sep=2,circle] at (4.0,0) {};
\node[fill=white] at (5.0,-1) {$\aleph_n$};
\node[fill=blue,inner sep=2,circle] (An) at (5.0,0) {};

\draw[blue,very thin,dashed] (4.5,0) -- (5.5,0);

\draw[thin] (6.0,-1) -- (8,-1);
\node[fill=white] at (7,-1) {$\aleph_\omega$};

\coordinate (OrOm) at (6.0,0);
\coordinate (TopOm) at ($(OrOm)+(60:1cm)$);
\coordinate (RtOm) at ($(OrOm)+(0:1cm)$);

\draw[very thin] (OrOm) -- (TopOm) -- (RtOm) -- cycle;
\node[fill=red,inner sep=1,circle] at (RtOm) {};
\node[fill=red,inner sep=1,circle] at (7.35,0) {};
\draw[very thin,red,dashed] (7.35,0) -- (8,0);
\node[fill=blue,inner sep=1,circle] at (8.0,0) {};

\node at (7.35,-0.45) {\small $\Delta+\omega+1$};

\node[fill=white] at (8.6,-1) {$\aleph_{\omega+1}$};
\node[fill=blue,inner sep=1,circle] at (8.5,0) {};
\draw[very thin,blue,dashed] (8.5,0) -- (9.0,0);

\coordinate (OrAlpha) at (9.25,0);
\coordinate (TopAlpha) at ($(OrAlpha)+(60:1cm)$);
\coordinate (RtAlpha) at ($(OrAlpha)+(0:1cm)$);

\draw[fill=gray!10,draw=white] (9.0,-1.25) -- (11.25,-1.25) -- (11.25,3) -- (9.0,3) -- cycle;

\draw[fill=gray!20,draw=white] (9.1,-0.75) -- (11.0,-0.75) -- (11.0,2.25) -- (9.1,2.25) -- cycle;

\node at (10.65,2.6) {all $\aleph_\alpha$};

\draw[very thin] (OrAlpha) -- (TopAlpha) -- (RtAlpha) -- cycle;
\node[fill=red,inner sep=1,circle] at (RtAlpha) {};
\draw[thin,red] (RtAlpha) -- (11,0);
\node[fill=blue,inner sep=2,circle] at (11.0,0) {};

\node at (10.05,-0.45) {\footnotesize $\Delta+\alpha$};

\node at (11.05,-0.45) {\footnotesize $+1$};

\node at (9.7,1.85) {\small only if};
\node at (10.05,1.45) {\small $\cof(\aleph_\alpha)=\omega$};

\draw[thin] (9.25,-1) -- (11.25,-1);
\node[fill=gray!10] at (10.25,-1) {\footnotesize $\aleph_\alpha$};

\node[fill=white] at (11.75,-1) {$\mathfrak{c}$};
\node[fill=blue,inner sep=2,circle] at (11.75,0) {};
\draw[blue,very thin,dashed] (11.25,0) -- (11.75,0);
\end{tikzpicture}
\end{figure}

It follows that the number of Tukey types of $\mathbf{seq}(M)$ lies between $\aleph_0$ and $\mathfrak{c}$, is equal to $\aleph_0$ precisely when $\mathfrak{c} < \aleph_{\omega_1}$, and is equal to $\mathfrak{c}$ if and only if $\mathfrak{c}$ is a fixed point of the aleph function and has uncountable cofinality.

\section{Preliminaries}

\subsection{Relations and Morphisms}
Any collection, $A$ say, of pairs, $(x,y)$, is called a \emph{general relation}. This definition allows for $A$ to be a proper class, for example $\in=\{(x,y) : x \in y\}$, or $\subseteq$.
The \emph{inverse} of $A$ is $A^{-1}=\{(x,y) : (y,x) \in A\}$. The \emph{complement} of $A$ is $\neg A=\{(x,y) : (x,y) \notin A\}$.
The \emph{dual} of $A$ is the complement of the inverse, $A^\perp=\neg A^{-1}=\{ (x,y) : (y,x) \notin A\}$.
We write $xAy$ if and only if $(x,y) \in A$, and we think of $A$ as a predicate, so `$xAy$' means `$A$ is true of $x$ and $y$'.
The \emph{domain} of $A$ is $\mathop{dom}(A)=\{x : xAy$ for some $y\}$ and the \emph{range} is, $\mathop{rng}(A)=\{y : xAy$ for some $x\}$.

A triple, $\mathbf{A}=(A_-,A_+, A)$ is a \emph{relation} if $A$ is a general relation, which we consider on $A_-\times A_+$.
Informally, we may think of $A_-$ as being a set of `challenges', $A_+$ being a set of `potential responses', and `$xAy$' means that $y$ correctly solves the problem, $x$.
For simplicity, we write $(A',A)$ for a relation $(A_-,A_+, A)$  where $A_-=A'=A_+$.
The \emph{dual}, $\mathbf{A}^\perp$, of $(A_-,A_+, A)$ is $(A_+,A_-, A^\perp)$.
A subset $C$ of $A_+$ is \emph{cofinal} if for every $x$ in $A_-$ there is a $y$ in $C$ so that $xAy$.
The \emph{cofinality} of a relation $\mathbf{A}$, denoted $\cof(\mathbf{A})$ is the minimal size of a cofinal set, if one exists, or $\infty$ otherwise.
A subset $U$ of $A_-$ is \emph{bounded by $y$} an element  of $A_+$ if $xAy$ for all $x$ in $U$. A subset is \emph{bounded} if it is bounded by something, and \emph{unbounded} if not bounded.
The \emph{additivity} of a relation $\mathbf{A}$, denoted $\add(\mathbf{A})$  is the minimal size of an unbounded set set, if one exists, or $\infty$ otherwise. Observe that $\add(\mathbf{A}) = \cof(\mathbf{A}^\perp)$.

Let $\mathbf{A}=(A_-,A_+,A)$ and $\mathbf{B}=(B_-,B_+,B)$ be two relations. A \emph{(Tukey) morphism} from the first to the second is a pair $( \psi_-,\phi_+)$ of functions such that (i) $\phi_+: A_+ \to B_+$, (ii) $ \psi_-: B_- \to A_-$ and
\[\text{(iii)}  \qquad  \psi_-(b) A a \implies b B \phi_+(a), \qquad \text{where } b \in B_- \ \text{and}\ a \in A_+.\]
When there is a morphism from relation $\mathbf{A}=(A_-,A_+,A)$ to $\mathbf{B}=(B_-,B_+,B)$ we write,  $(A_-,A_+,A) \gtq (B_-,B_+,B)$ or just $\mathbf{A} \gtq \mathbf{B}$.

\begin{lem}\label{l:gtq_dual}
If $\mathbf{A} \gtq \mathbf{B}$ say via morphism $( \psi_-,\phi_+)$ then $\mathbf{B}^\perp \gtq \mathbf{A}^\perp$ via the morphism $(\phi_+, \psi_-)$.
\end{lem}

\begin{lem}
 If $\mathbf{A} \gtq \mathbf{B}$ via $(\psi_-^0,\phi_+^0)$   and $\mathbf{B} \gtq \mathbf{C}$ via $(\psi_-^1,\phi_+^1)$ then $\mathbf{A} \gtq \mathbf{C}$ via $(\psi_-^0 \circ \psi_-^1, \phi_+^1  \circ \phi_+^0)$.
\end{lem}

It follows that $\gtq$ is a transitive relation on relations.
Say that two relations, $\mathbf{A}$ and $\mathbf{B}$, are \emph{bimorphic} if $\mathbf{A} \gtq \mathbf{B}$ and $\mathbf{B} \gtq \mathbf{A}$. This is denoted $\mathbf{A} \gte \mathbf{B}$.
Now, modulo bimorphism, $\gtq$ is a partial order on relations.

\begin{lem}
If $( \psi_-,\phi_+)$ is a morphism from $\mathbf{A}$ to $\mathbf{B}$ then $\phi_+$ carries cofinal sets of $\mathbf{A}$ to cofinal sets of $\mathbf{B}$, while $ \psi_-$ carries unbounded sets of $\mathbf{B}$ to unbounded sets of $\mathbf{A}$.
\end{lem}

Call a function $ \psi_-:B_- \to A_-$ a \emph{lower morphism} of $\mathbf{A}$ to $\mathbf{B}$ if it carries unbounded sets of $\mathbf{B}$ to unbounded sets of $\mathbf{A}$.
Call a function $\phi_+:A_+ \to B_+$ an \emph{upper morphism} if takes cofinal sets of $\mathbf{A}$ to cofinal sets of $\mathbf{B}$.

\begin{lem}\label{l:ulm} \

(1) If $ \psi_-$ is a lower morphism then there is a $\phi_+ : A_+ \to B_+$ such that $( \psi_-,\phi_+)$ is a morphism $\mathbf{A} \gtq \mathbf{B}$.

(2) If $\phi_+$ is an upper  morphism then there is a $ \psi_- : B_- \to A_-$ such that $( \psi_-,\phi_+)$ is a morphism $\mathbf{A} \gtq \mathbf{B}$.
\end{lem}

\subsection{Operations on Relations}\label{ss:sops} In addition to duality, there are a variety of operations on relations.

Suppose we have relations $\mathbf{A}=(A_-,A_+,A)$ and  $\mathbf{B}=(B_-,B_+,B)$ then $\mathbf{A} \times \mathbf{B}$ is the standard product, $(A_-\times B_-,A_+\times B_+,A\times B)$, while $\mathbf{A} \with  \mathbf{B}$, read `$\mathbf{A}$ with $\mathbf{B}'$, is the relation $(A_- \oplus B_-,A_+ \times B_+,W)$ where $xW(y_1,y_2)$ if and only if $x$ is in $A_-$ and $xAy_1$ or $x$ is in $B_-$ and $xBy_2$.
Clearly $\mathbf{A} \times \mathbf{B} \gtq \mathbf{A} \with \mathbf{B}$. The key observation about the operation  `with' ($\with$) is that it gives the least upper bound, with respect to $\gtq$, of a pair of relations.
\begin{lem}\label{l:with}
    For any three relations $\mathbf{A}, \mathbf{B}$ and $\mathbf{C}$, we have $\mathbf{A} \gtq \mathbf{B} \with  \mathbf{C}$ if and only if $\mathbf{A} \gtq \mathbf{B}$ and $\mathbf{A} \gtq \mathbf{C}$.
\end{lem}

Given $\mathbf{A}$ and $\mathbf{B}$ define $\mathbf{A} + \mathbf{B}$,  to be $(A_- \oplus B_-,A_+ \oplus B_+,S)$ where $x S y$ if and only if either $x \in A_-$, $y \in A_+$ and $xAy$ or $x\in B_-$, $y\in B_+$ and $xBy$.
More generally, if $\{\mathbf{A}_\lambda=((A_\lambda)_-,(A_\lambda)_+,A_\lambda) : \lambda \in \Lambda\}$ is a family of relations then its sum is, $\sum_{\lambda \in \Lambda} \mathbf{A}_\lambda = (\bigoplus_{\lambda \in \Lambda} (A_\lambda)_-,\bigoplus_{\lambda \in \Lambda} (A_\lambda)_+, S)$ where $x S y$ if and only if for some $\lambda$, $x\in (A_\lambda)_-, (y \in (A_\lambda)_+$ and $xA_\lambda y$.

Note:

\begin{lem}\label{big_sum=product} If for all $\lambda$, $\mathbf{A}_\lambda = \mathbf{A}$ then $\sum_{\lambda \in \Lambda} \mathbf{A}_\lambda \gte \mathbf{A} \times (|\Lambda|,=)$.
\end{lem}

\begin{lem}\label{l:sum1}
    Let $\{\mathbf{A}_\lambda : \lambda \in \Lambda\}$ be a family of relations, and $\Lambda'$ a subset of $\Lambda$. Then $\sum_{\lambda \in \Lambda} \mathbf{A}_\lambda \gtq \sum_{\lambda \in \Lambda'} \mathbf{A}_\lambda$.
\end{lem}

\begin{lem}\label{l:sum2}
    If for every $\lambda$ in $\Lambda$,  $\mathbf{A}_\lambda \gtq \mathbf{B}_\lambda$, then $\sum_{\lambda \in \Lambda} \mathbf{A}_\lambda \gtq \sum_{\lambda \in \Lambda} \mathbf{B}_\lambda$.
\end{lem}

\begin{lem} $\mathbf{A} + \mathbf{B} \gtq \mathbf{A} \with \mathbf{B}$.
\end{lem}
\begin{proof} By Lemma~\ref{l:with} and symmetry it suffices to show $\mathbf{A}+\mathbf{B} \gtq \mathbf{A}$. We use the identity map for $ \psi_-$,  define $\phi_+(y) = y$ if $y\in A_+$, and otherwise set $\phi_+(y)$ to be an arbitrary element in $A_+$. Then for any $x\in A_-$, $ \psi_-(x)S y$ implies that $y\in A_+$. Therefore $x A y$ which means $xA\phi_+(y)$.    \end{proof}

\subsection{Some Special Relations}
\label{ss:ssr}

Define, first, $\mathbf{1}=(1,=)$ and $\mathbf{0}=(\emptyset,\emptyset,\emptyset)$ (the relation with empty domain).
Now for $f$ and $g$ from $\omom$, let $f\le_\infty g$ if and only if there is an infinite subset $A$ of $\omega$ such that $f(n) \le g(n)$ for all $n$ in $A$ (so, $f$ is less than or equal to $g$ infinitely often).
Note $\le_\infty$ is $\not>^*$, and so  $(\omom,\le^*)^\perp = (\omom,\not>^*)=(\omom,\le_\infty)$.
Combining this with the easily verified fact that $\cof(\mathbf{A} \with \mathbf{B})=\max(\cof(\mathbf{A}),\cof(\mathbf{B}))$ and $\cof(\mathbf{A} \times \mathbf{B})=\cof(\mathbf{A})\cdot\cof(\mathbf{B})$, we can compute the cofinality and additivity of certain special relations.
\begin{lem}\label{l:add_cof_omom_infty} \

(1) The following are all equal: $\mathfrak{b}$, $\cof(\omom,\le_\infty)$, $\add (\omom,\le^*)$,  $\cof((\omom,\le_\infty)\with \omega)$ and $\cof((\omom,\le_\infty)\times \omega)$.

(2) The following are all equal: $\mathfrak{d}$, $\cof (\omom,\le^*)$, $\add(\omom,\le_\infty)$, $\add((\omom,\le_\infty)\with  \omega)$ and
$\add((\omom,\le_\infty)\times \omega)$.
\end{lem}

\begin{lem}\label{l:cal_omom_infty} Let regular $\kappa \ge \lambda$ be infinite cardinals.

Then
$(\omom,\le_\infty)$ is not
calibre $(\kappa,\lambda)$
if and only if $\lambda=\mathfrak{b}=\mathfrak{d}=\kappa$.
\end{lem}
\begin{proof}
Since $(\omega^\omega,\le_\infty)$  has additivity $\mathfrak{d}$ it is automatically calibre $(\kappa,\lambda)$ for any $\kappa \ge \lambda$ when $\lambda < \mathfrak{d}$. So it suffices to show that  $(\omega^\omega,\le_\infty)$ is not calibre $\kappa$  if and only if $\kappa=\mathfrak{b}=\mathfrak{d}$.
But $(\omega^\omega,\le_\infty)$ is not calibre $\kappa$ if and only if $(\omega^\omega,\le_\infty) \gtq \kappa$. Now since $\kappa$ is regular we have  $\mathfrak{d}=\add(\omega^\omega,\le_\infty) \le \add(\kappa)=\kappa$ and $\mathfrak{b}=\cof(\omega^\omega,\le_\infty) \ge \cof(\kappa)=\kappa$. Hence $\kappa=\mathfrak{b}=\mathfrak{d}$.
Conversely, if $\mathfrak{b}=\mathfrak{d}$ then let $\{f_\alpha : \alpha < \mathfrak{b}\}$ be a scale. Then for every $\mathfrak{b}$-sized subset, $A$ say, of $\mathfrak{b}$ we see that $\{f_\alpha : \alpha \in A\}$ is $<^*$-cofinal, and so $\le_\infty$-unbounded. Thus $(\omega^\omega,\le_\infty)$ is not calibre $\mathfrak{d}$.
\end{proof}

\begin{lem}\label{l:with_vs_and}
    We have $(\omom,\le_\infty) \with \omega \not\gtq (\omom,\le_\infty) \times \omega$.
\end{lem}
\begin{proof}
    Consider the following property $(\ast)$ of a relation $\mathbf{A}$: there is a $\sigma$-bounded subset $T$ of $A_-$ such that if $S$ is a countable subset of $A_-$ disjoint from $T$ then $T$ is bounded. It is straight forward to check that if $\mathbf{A} \gtq \mathbf{B}$ and $\mathbf{A}$ satisfies $(\ast)$ then so does $\mathbf{B}$.

    Observe that, taking $T$ to be the copy of $\omega$ in the domain of $(\omom,\le_\infty) \with \omega$, and recalling that $(\omom,\le_\infty)$ is countably directed, $(\omom,\le_\infty) \with \omega$ has $(\ast)$. But $(\omom,\le_\infty) \times  \omega$ does not have $(\ast)$. To see this take any subset $T=\bigcup_n T_n$ of the domain of $(\omom,\le_\infty) \times \omega$ where each $T_n$ is bounded, say by $(f_n,m_n)$. Pick $g$ so that $f_n <^* g$ for all $n$. Let $S=\{(g,n) : n \in \omega\}$. Then $S$ is disjoint from $T$ but is not bounded in $(\omom,\le_\infty) \times \omega$.
\end{proof}

Let $(\kappa_n)_n$ be a sequence of cardinals. Let $\left(\prod_n [\kappa_n]^{<\omega}\right)_\infty$ be all $(F_n)_n$ in the product such that $F_n\ne \emptyset$ for infinitely many $n$.
Define a relation $i_\infty$ with domain and range $\left(\prod_n [\kappa_n]^{<\omega}\right)_\infty$ by $(F_n)_n i_\infty (G_n)_n$ if and only if for infinitely many $n$ we have $F_n$ meeting $G_n$, in other words, there is an infinite subset $A$ of $\omega$, such that for all $n$ in $A$, $F_n \cap G_n \ne \emptyset$.

\begin{lem}\label{l:small_Pk} Let $(\kappa_n)_n$ be a sequence of cardinals, all non-zero.

(1) $\left(\left(\prod_n [\kappa_n]^{<\omega}\right)_\infty,i_\infty \right) \gte \mathbf{1}$ if and only if all but finitely many of the $\kappa_n$ are finite.

(2) $\left(\left(([\omega]^{<\omega})^\omega\right)_\infty,i_\infty \right) \gte (\omom,\le_\infty)$.
    \end{lem}
\begin{proof} The proof of (1) is straightforward.

For (2), let $\mathbf{A}=\left(\left(([\omega]^{<\omega})^\omega\right)_\infty,i_\infty \right)$ and $\mathbf{B}=(\omom,\le_\infty)$.
We first show $\mathbf{A} \gtq \mathbf{B}$.
Define $\phi_+((F_n)_n)(m)=\max F_m$ (taking $\max \emptyset=0$).
Define $ \psi_-(f)=(F_n)_n$ where $F_n=\{f(n)\}$. Note $\phi_+$ and $ \psi_-$ are well-defined.
Suppose $ \psi_-(f) i_\infty (G_n)_n$. Then there is an infinite subset $A$ of $\omega$ such that for every $m$ from $A$ we have $f(m) \in G_m$.
But now we see that for every $m$ in $A$, $\phi_+((G_n)_n)(m) \ge \max G_m \ge f(m)$, and so $f \le_\infty \phi_+((G_n)_n)$, as required for $( \psi_-,\phi_+)$ to be a morphsim.

Now we show $\mathbf{B} \gtq \mathbf{A}$.
Define $\psi_+:\omom \to \left(([\omega]^{<\omega})^\omega\right)_\infty$ by $\psi_+(f)=(F_n)_n$ where $F_n=[0,\max (f(0),\ldots,f(n))]$, noting this is well-defined.
Now take any $(F_n)_n$ from $\left(([\omega]^{<\omega})^\omega\right)_\infty$. Let $(k_m)_m$ enumerate in strictly increasing order all $k$ such that $F_k \ne \emptyset$, and note $k_m \ge m$.
Define $\psi_-((F_n)_n)(m)=\max F_{k_m}$.
Suppose that $\psi_-((F_n)_n) \le_\infty g$. Then there is an infinite subset $A$ of $\omega$ such that for every $m$ in $A$ we have $\max F_{k_m} \le g(m)$.
Now for every $m$ from $A$, $\max F_{k_m} \le g(m) \le \max (g(0),\ldots,g(m),\ldots, g(k_m))$, and so $\emptyset \ne F_{k_m} \subseteq [0,\max F_{k_m}] \subseteq \psi_+(g)(k_m)$.
Hence $(F_n)_n i_\infty \psi_+(g)$, as required for $(\psi_-,\psi_+)$ to be a morphism.
\end{proof}

Let $P(\kappa)=((([\kappa]^{<\omega})^\omega)_\infty, i_\infty)$.
Set $Q(\kappa)=\mathbf{0}$ and $S(\kappa)=\mathbf{0}$ if $\cof(\kappa) \ne \omega$.
Otherwise select a sequence, $(\kappa_n)_n$, strictly increasing up to $\kappa$, and let $Q(\kappa)=(( \prod_n [\kappa_n]^{<\omega}    )_\infty,i_\infty)$, and $S(\kappa)=\sum_n P(\kappa_n) \times (\kappa_n,=)$.
The next lemma tells us that $Q(\kappa)$ and $S(\kappa)$ are well-defined in the case when $\kappa$ has countable cofinality, because their  Tukey type is independent of the choice of cofinal sequence, $(\kappa_n)_n$.
\begin{lem} Suppose $(\kappa_n)_n$ and $(\kappa_n')_n$ are increasing sequences of cardinals converging to cardinal $\kappa$ (of countable cofinality).

    Then $(( \prod_n [\kappa_n]^{<\omega}    )_\infty,i_\infty) \gte (( \prod_n [\kappa'_n]^{<\omega}    )_\infty,i_\infty)$ and $\sum_n P(\kappa_n) \times (\kappa_n,=) \gte \sum_n P(\kappa'_n) \times (\kappa'_n,=)$.

\end{lem}
\begin{proof}
    This is clear from Lemmas~\ref{l:sum1} and~\ref{l:sum2} because the sequences $(\kappa_n)_n$ and $(\kappa_n')_n$ are interleaved.
\end{proof}

\subsection{Topology: Notation, and a Special Space}
Our topological definitions and notation are standard, see \cite{Eng}. However we give concrete names and notations to certain objects. Recall $\mathop{NC}(X)$ is the collection of all non-closed subsets of a space $X$, and $\mathop{CS^+}(X)$ is all infinite convergent sequences with limit.
Additionally, let $\mathop{CS}(X)$ be all infinite convergent sequences without limit  and
(in other words, $\mathop{CS}(X)=\{e(\omega) : e $ is an embedding of $\omega+1$ into $X\}$ and $\mathop{CS^+}(X)=\{e(\omega+1) : e $ is an embedding of $\omega+1$ into $X\}$).
For any $S$ in $\mathop{CS}(X)$ set $S^+$ to be $S$ along with its limit, so $S^+$ is in $\mathop{CS^+}(X)$.

One space plays a key role throughout this paper.
The metric fan, $\mathbb{F}$, has underlying set $(\omega \times \omega) \cup \{\infty\}$, and topology where points of $\omega \times \omega$ are isolated while basic neighborhoods of $\infty$ have the form $\{\infty\} \cup ([n,\infty) \times \omega)$, for $n \in \omega$.
The infinite convergent sequences are those subsets (i) containing $\infty$ (which is the limit), (ii) having finite intersection with every vertical line $\{n\} \times \omega$, and (iii)  infinitely many intersections with vertical lines are non-empty.
The compact subsets of $\mathbb{F}$ are those that are finite, along with the infinite convergent sequences.
Clearly, the metric fan is not locally compact. Further, it detects whether suitably nice spaces are locally compact.
\begin{lem}
    A separable metrizable space $M$ is not locally compact if and only it it contains a closed copy of $\mathbb{F}$.
\end{lem}

\section{The $k$-Structure}\label{s:k}

\subsection{General Results}

We start by recording the impact of basic topological operations on the $k$-structure. The proof is left to the reader.

\begin{lem} Let $X$ and $Y$ be spaces.
 Then $\mathbf{k}(X \oplus Y) \gte \mathbf{k}(X) \with \mathbf{k}(Y)$. Further, if $Y$ is a closed subspace of $X$ or $Y$ is a quotient of $X$ then $\mathbf{k}(X) \gtq \mathbf{k}(Y)$. 
\end{lem}

We now specialize to sequential spaces, noting that our main results below are for separable metrizable spaces, which are, indeed, sequential.
In a sequential space we can identify a small subset, namely $\mathop{CS}(X)$, of $\mathop{NC}(X)$ to which  we can restrict attention. Below, for sequential spaces,  we use $(\mathop{CS}(X),\K(X),\nci)$ interchangeably with $\mathbf{k}(X)$.
\begin{lem} For $X$ sequential,  $\mathbf{k}(X) \gte (\mathop{CS}(X),\K(X),\nci)$.
\end{lem}
\begin{proof}
To see    $(\mathop{NC}(X),\K(X),\nci) \gtq (\mathop{CS}(X),\K(X),\nci)$ just take $\phi_+$ and $ \psi_-$ to be the identity maps.
For the converse, let $\phi_+$ be the identity map. Define $ \psi_-: \mathop{NC}(X) \to \mathop{CS}(X)$ by $ \psi_-(F)=S$ where for the non-closed set $F$ we have picked an infinite sequence $(x_n)_n$ on $F$ converging to a point outside $F$, and set $S=\{x_n\}_n$.
Now $ \psi_-, \phi_+$  clearly form a morphism.
\end{proof}

\begin{lem}\label{l:below_K}
For any space $X$ we have $\K(X) \gtq (\mathop{CS}(X),\K(X),\nci)$.
Hence, if $X$ is sequential then $\K(X) \gtq \mathbf{k}(X)$.
\end{lem}
\begin{proof}
Define $\phi_+(K)=K$ and $ \psi_-(S)=\cl{S}$. Noting that the closure of a convergent sequence without limit point is the sequence \emph{with} the limit, which is compact, we see $\phi_+$ and $ \psi_-$ are well defined.
Take any  $S$ from $\mathop{CS}(X)$. If $ \psi_-(S)=\cl{S} \subseteq K$ then $S \cap K \ne \cl{S} \cap K$, so $S$ is not closed in $K=\phi_+(K)$,  $S(\nci)\phi_+(K)$, as required for $\phi_+$ and $ \psi_-$ to be a morphism.
\end{proof}

\begin{lem}\label{l:csp}
    Let $X$ be sequential. Then $(\mathop{CS^+}(X),\K(X),\subseteq) \gtq \mathbf{k}(X)$.
\end{lem}
\begin{proof}
Set $\phi_+$ to be the identity and $ \psi_-(S)=S^+$. If $ \psi_-(S) \subseteq K$, so $S^+ \subseteq K$,  then $S\cap K$ is not closed in $K$, so $S\nci\phi_+(K)$.
\end{proof}

Now we specialize further to separable metrizable spaces. Let $M$ be separable metrizable.
Write $I(M)$ for all the isolated points of $M$, and $M'$ for $M \setminus I(M)$.
\begin{lem}\label{l:above_kc}
    If $M$ is separable metrizable where $M'$ not compact then $\mathbf{k}(M) \gtq \mathbf{kc}(M)$.
\end{lem}
\begin{proof}
Write $I(M)=\{z_n\}_{n \in \omega}$, and fix an infinite closed discrete subset, $E=\{e_n\}_{n \in \omega}$ of $M'$. Define $\phi^+(K) = K \cup Z(K)$ where $Z(K)=\{z_n :  e_n \in K\}$. Define $ \psi_-(z_n)$ to be a member of $\mathop{CS}(M)$ converging to $e_n$, and for $x$ in $M'$ set $ \psi_-(x)$ to be a  member of $\mathop{CS}(M)$ converging to $x$.
Observe that $\phi_+$ and $ \psi_-$ are well defined.

Now suppose $ \psi_-(x)\nci K$. We show $x\in \phi_+(K)$. Note $ \psi_-(x)$ is a sequence converging to an element of $K$. If $x=z_n$, for some $n$, then the limit is $e_n$, and from $e_n \in K$ we have $x=z_n \in Z(K) \subseteq \phi_+(K)$.
On the other hand, if $x \in M'$ then the limit is $x$, and so $x \in K \subseteq \phi_+(K)$.
\end{proof}

\begin{thm}\label{th:kc^om_above_k}
    Let $M$ be separable metrizable.
    Then     $\mathbf{kc}(M)^\omega  \gtq \mathbf{k}(M)$.
\end{thm}
\begin{proof} We take $\mathbf{k}(M)=(\mathop{CS}(X),\K(X),\nci)$.
Let $\gamma M$ be a metrizable compactification of $M$. Fix a compatible metric $d$, and select a base, $\mathcal{B}=\{B_m : m \ge 1\}$ for $\gamma M$ where $B_1=\gamma M$, $\mathop{diam} B_m \to 0$ and for every $x$ in $\gamma M$ the set $\{n : x \in B_n\}$ is infinite.

For any sequence $\mathbb{K}=(K_m)_{m \in \omega}$ from $\K(M)^\omega$ define
\[\phi_+(\mathbb{K}) = K_0 \cup \bigcup \{ K_m \cap \cl{B_m} : m\ge 1 \ \& \ K_0 \cap \cl{B_m} \ne \emptyset\}.\]
Then  $\phi_+(\mathbb{K})$ is compact. To see this
take any cover by basic open sets. Some finite subcollection, say $\mathcal{U}_1$, covers the compact set $K_0$.
As $K_0$ and $\gamma M \setminus \bigcup \mathcal{U}_1$ are compact and disjoint, the distance between them is strictly positive. As $\mathop{diam} B_m \to 0$, it follows that,  of the $\cl{B_m}$ that meet $K_0$, all but finitely many  are contained in $\bigcup \mathcal{U}_1$. Hence $\phi_+(\mathbb{K}) \setminus \bigcup  \mathcal{U}_1$ is compact, and can be covered with a finite subcollection, say $\mathcal{U}_2$, from the basic open cover. Now $\mathcal{U}_1 \cup \mathcal{U}_2$ is a finite subcover of $\phi_+(\mathbb{K})$.

For  any sequence $S=(x_n)_n$  from $\mathop{CS}(M)$, say  converging to $x$, define $\psi_-(S)$ in $M^\omega$ as follows.
From the base $\mathcal{B}$ pick a local base at $x$, say $(B_{m_n})_{n\ge 1}$,  where $(m_n)_n$ is strictly increasing, $m_1=1$ and  $\cl{B_{m_{n+1}}} \subseteq B_{m_n}$ for all $n$.
For each $n$, as $(x_n)_n$ converges to $x$, there are only finitely many, say $F_{m_n}$, terms of the sequence $S$ in $\cl{B_{m_n}} \setminus B_{m_{n+1}}$.
Let $M=\{m_n : F_{m_n} \ne \emptyset\}$, and note it is infinite but does not contain $0$.
For $m$ not in $M$ set $y_m=x$. Otherwise set $y_{m_n}$ to be some element from $F_{m_n}$.
Set $\psi_-(S)=(y_m)_m$.

Now suppose $\psi_-(S) \in^\omega (K_n)_n$. So for all $n$, have $y_n \in K_n$. We show $S$ has infinite intersection with $\phi_+( (K_n)_n)$.
Note first that $x=y_0$ is in $K_0$, so for all $n$ we have $K_0 \cap B_{m_n} \ne \emptyset$. Second, for each of the infinitely many $y_{m_n}$ from $M$, we have $y_{m_n} \in K_{m_n} \cap F_{m_n} \subseteq K_{m_n} \cap \cl{B_{m_n}}$. Together  it follows the infinite set $\{y_{m_n} : m_n \in M\} \subseteq S \cap \phi_+((K_n)_n)$.
\end{proof}

\subsection{Initial Structure of $\mathbf{k}(\mathcal{M})$}
Recall by $\mathbf{k}(\mathcal{M})$ we mean the  Tukey types (equivalence classes) of $\mathbf{k}(M)$, where $M$ is separable metrizable, ordered by $\gtq$.
Further recall, for separable metrizable $M$, that $I(M)$ is all isolated points of $M$, and $M'$ is $M \setminus I(M)$.
 Of course $M'$ is closed in $M$, and so $\sigma$-compact if $M$ is $\sigma$-compact.
Write $\mathop{LC}(M)$ for $\{x \in M : x$ has a compact neighborhood$\}$, and $M^\#$ for $M \setminus \mathop{LC}(M)$. Then $M$ is locally compact if and only if $M^\# \ne \emptyset$, $M^\#$ is a closed subset of $M'$, and if $M^\#$ is not compact then $M'$ is not compact.

\begin{thm}\label{th:lb_k} Let $M$ be separable metrizable.

(1) $\mathbf{k}(M) \gtq \omega$ $\iff$ $M'$ not compact.

(2) $\mathbf{k}(M) \gtq (\omom,\le_\infty)$ $\iff$ $M$ not locally compact $\iff$ $M^\#\ne \emptyset$.

(3) $\mathbf{k}(M) \gtq (\omom,\le_\infty) \with  \omega$ $\iff$ $M$ not locally compact and $M'$ not compact $\iff$ $M^\# \ne \emptyset$ and $M'$ not compact.

(4)  $\mathbf{k}(M) \gtq (\omom,\le_\infty) \times \omega$ $\iff$ $M$ not locally compact, $M'$ not compact and $M^\#$ not compact $\iff$ $M^\#$ not compact.
\end{thm}
\begin{proof} First note that claim (3) follows from claims (1) and (2), via Lemma~\ref{l:with}.

\smallskip

\noindent {\textbf{For (1):}}
Suppose $M'$ is compact. Take any $ \psi_-:\omega \to \mathop{CS}(M)$. We show it is not a lower morphism, and so $\mathbf{k}(M) \not\gtq \omega$, as required.
For each $n$ in $\omega$, $S_n= \psi_-(n)$ is an infinite convergent sequence, say with limit $x_n$, which must be in $M'$. As $M'$ is compact there is an infinite subset, $A$ say, of $\omega$, such that $\{x_n : n \in A\}$ is a convergent sequence. We can select for each $n$ a tail, $T_n$, of $S_n$ so that $\bigcup_{n \in A} T_n$ has compact closure, say $K$.
Now $A$ is unbounded in $(\omega,\le)$ but $ \psi_-(A)$ is bounded in $\mathbf{k}(M)$, because for each $n$ in $A$ we have $ \psi_-(n)=S_n \nci K$ as $T_n \subseteq S_n \cap K$ is infinite.

For the reverse implication, note if $M'$ is not compact, then $M$ is not compact so, $\mathbf{kc}(M) \gtq \omega$, and then by Lemma~\ref{l:above_kc} $\mathbf{k}(M) \gtq \mathbf{kc}(M) \gtq \omega$.

\smallskip

\noindent {\textbf{For (2):}}
Suppose $M$ is locally compact. Then $\omega \gtq \mathbf{k}(M)$. But $\omega  \not\gtq (\omom,\le_\infty)$ (Lemma~\ref{l:add_cof_omom_infty}), so $\mathbf{k}(M) \not\gtq (\omom,\le_\infty)$.

Now suppose $M$ is not locally compact. Then $\mathbb{F}$ embeds as a closed subset in $M$, and hence $\mathbf{k}(M) \gtq \mathbf{k}(\mathbb{F})$.
So it suffices to show $\mathbf{k}(\mathbb{F}) \gtq (\omega^\omega,\le_\infty)$.
Define $\phi_+$ by $\phi_+(K)(n)= \max \{m : (n,m) \in K\}$ (where $\max \emptyset=0$) and $ \psi_-(f)=\{(n,f(n)) : n \in \omega\}$. Note these are well defined.
Take any $f$ in $\omega^\omega$. Suppose $ \psi_-(f)\nci K$. Then there is an infinite subset $A$ of $\omega$ such that for all $n$ in $A$, $(n,f(n)) \in K$. But if $n$ is from $A$, and so $(n,f(n)) \in K$, then $f(n) \le \phi_+(K)(n)$. Hence $f \le_\infty \phi_+(K)$, as required for $( \psi_-,\phi_+)$ to be a morphism.

\smallskip

\noindent {\textbf{For (4):}} Suppose, first, that  $M^\#$ is compact. As $\mathop{LC}(M)$ is always $\sigma$-compact, and $M$ is the union of (compact) $M^\#$ and $\mathop{LC}(M)$, we see  that $M$ is $\sigma$-compact, and hence  $M'$ is $\sigma$-compact. So  by  Theorem~\ref{th:ub_k}(3) (the reverse direction, for which a self-contained proof is given below) $(\omom,\le_\infty)\with  \omega \gtq \mathbf{k}(M)$. But by Lemma~\ref{l:with_vs_and},  $(\omom,\le_\infty)\with  \omega \not\gtq (\omom,\le_\infty) \times \omega$. Hence $\mathbf{k}(M) \not\gtq (\omom,\le_\infty) \times \omega$, as needed for the forward direction of (4).

To complete the argument for (4), suppose
 $M^\#$ is not compact. Then we can embed $\mathbb{F} \times \omega$ as a closed subspace in $M$. So it suffices to show $\mathbf{k}(\mathbb{F} \times \omega) \gtq (\omom,\le_\infty) \times \omega$. Define $\phi_+:\K(\mathbb{F}\times \omega) \to \omom \times \omega$ as follows. Take any compact subset $K$ of $\mathbb{F}\times \omega$. Let $N_K$ be the largest $n$ such that $K$ meets $\mathbb{F} \times \{n\})$. Define $f_K(m)$ to be the largest $p$ over all $n \le N_K$ such that $(m,p,n)$ is in $K \cap \left(\mathbb{F} \times \omega\right)$. Set $\phi_+(K)=(f_K,N_K)$, and note this is well-defined. Define $ \psi_-:\omom \times \omega \to \mathop{CS}(\mathbb{F} \times \omega)$ by $ \psi_-(f,n)=\{(m,f(m),n) : m \in \omega\}$, in other words $\mathop{Gr}(f) \times \{n\}$, where $\mathop{Gr}(f)$ is the graph of $f$.

Suppose that $ \psi_-(f,n) \nci K$. Then $ \psi_-(f,n) \cap K$ is infinite. So there is an infinite $A\subseteq \omega$ such that $\{(m,f(m),n) : m \in A\} \subseteq K$. It follows that $n \le N_K$, and for each $m$ in $A$ that $f(m) \le f_K(m)$, which means $f \le_\infty f_K$. Hence $(f,n) \left(\le_\infty \times \le\right) \phi_+(K)$, as required for $( \psi_-,\phi_+)$ to be a morphism.
\end{proof}

\begin{thm}\label{th:ub_k} Let $M$ be separable metrizable.

(1) $\omega \gtq \mathbf{k}(M)$ $\iff$ $M$ locally compact $\iff$ $M^\#=\emptyset$.

(2) $(\omom,\le_\infty) \gtq \mathbf{k}(M)$ $\iff$ $M'$ compact or empty.

(3) $(\omom,\le_\infty)\with  \omega \gtq \mathbf{k}(M)$ $\iff$ $M'$ is $\sigma$-compact or empty and $M^\#$ compact or empty.

(4) $(\omom,\le_\infty)\times \omega \gtq \mathbf{k}(M)$ $\iff$ $M'$ $\sigma$-compact or empty.
\end{thm}
\begin{proof} Claim (1) is clear because $\omega \gtq \mathbf{k}(M)$ if and only if $\mathbf{k}(M)$ has countable cofinality, which means it is hemicompact, and a first countable space is hemicompact if and only if it is Lindel\"{o}f and locally compact.

\noindent {\textbf{For (2):}}
For the forward direction of (2), suppose $M'$ is not compact (so not empty). Then $M$ is not compact, and $\mathbf{kc}(M) \gtq \omega$. By Lemma~\ref{l:above_kc} $\mathbf{k}(M) \gtq \omega$. But $(\omom,\le_\infty) \not\gtq \omega$ (Lemma~\ref{l:add_cof_omom_infty}).

For the reverse direction, suppose $M'$ is compact or empty. We show $(\omega^\omega,\le_\infty) \gtq \mathbf{k}(M)$. Indeed if $M'$ is empty then $\mathbf{k}(M)$ is trivial, and there is nothing to do. So suppose $M'$ is compact. If  $M$ were locally compact then it would be the disjoint sum of a compact subspace and a discrete  subspace, and certainly $(\omega^\omega,\le_\infty) \gtq \mathbf{1} \gte \mathbf{k}(M)$. So we suppose $M$ not locally compact.
Now we can find a strictly decreasing local base around $M'$, say $(C_n)_{n \in \omega}$, with $C_0=M$ and such that (using $M$ not locally compact) for each $n$ the set $E_n = C_n\setminus C_{n+1}$ is infinite, closed and discrete. For each $n$ fix a bijection $e_n : \omega \to E_n$.

Define $\phi_+(f)= M' \cup \bigcup_{n \in \omega} e_n( [0,\max \{f(0),\ldots,f(n)\}])$.
Note that $\phi_+(f)$ is compact and so $\phi_+$ maps from $\omega^\omega$ into $\K(M)$. Now we define $ \psi_-:\mathop{CS}(M) \to \omega^\omega$. Take any
 $S$ in $\mathop{CS}(\mathbb{F})$. Note that the limit of $S$ is in $M'$ (all other points being isolated).
 If $S \cap M'$ is infinite then set $ \psi_-(S)$ to be anything in $\omega^\omega$. Otherwise,
for each $n \in \omega$, define $n^+$ to be the first element of $\omega$ which is greater than or equal to $n$ and $E_{n+} \cap S \ne \emptyset$, and define $ \psi_-(S)(n) = \max \{m : e_{n^+}(m) \in E_{n^+} \cap S\}$. Note that, as $E_{n^+}$ is closed discrete while $S$ is a sequence converging to some point in $M'$, the set $E_{n^+} \cap S$ is finite, and $ \psi_-$ is well defined.

Take any $S$ in $\mathop{CS}(\mathbb{F})$. Suppose $ \psi_-(S) \le_\infty f$. So there is an infinite $A$ such that for every $n$ from $A$ we have $ \psi_-(S)(n) \le f(n)$. If $S \cap M'$ is infinite then, as $M' \subseteq \phi_+(g)$ for every $g$, we have $S\nci\phi_+(f)$. Otherwise,
 for every $n$ in $A$ we have $E_{n^+} \cap S \ne \emptyset$.
 Take any $n$ in $A$. Then
  $\max \{m : e_{n^+}(m) \in E_{n^+} \cap S\}= \psi_-(S)(n) \le f(n) \le \max \{f(0),\ldots,f(n^+)\}$. Thus $e_{n^+}( \psi_-(S)(n)) \in E_{n^+} \cap S \cap \phi_+(f)$.  Hence $S\nci \phi_+(f)$, as required for $( \psi_-,\phi_+)$ to be a morphism.

\medskip

\noindent {\textbf{For (4):}} For the forward direction of (4), suppose $(\omom,\le_\infty)\times \omega \gtq \mathbf{k}(M)$.
    If $M'$ is compact or empty then we are done.   Otherwise, by Lemma~\ref{l:above_kc}, we know $\mathbf{k}(M) \gtq \mathbf{kc}(M)$.
    Fix a lower morphism $ \psi_- : M \to \omega^\omega \times \omega$, witnessing $(\omega^\omega,\le_\infty) \times \omega \gtq \mathbf{kc}(M)$.

    Fix, for the moment, $n \in \omega$. Let $M_n= \psi_-^{-1} \left(\omega^\omega \times \{n\}\right)$ and pick a countable dense subset $D_n$ of $M_n$. Note that $\omega^\omega \times \{n\}$ under $\le_\infty$ has additivity $\mathfrak{d} \ge \omega_1$, so $ \psi_-(D_n)$ is bounded. Thus $D_n$ is bounded in $\mathbf{kc}(M)=(M,\K(M),\in)$, which means its closure, $K_n$ say, is compact.
    Since $M=\bigcup_n M_n$ we see that $M=\bigcup_n K_n$ is  $\sigma$-compact, as claimed.

For the reverse direction, suppose $M'$ is $\sigma$-compact or empty. In the latter case $\mathbf{k}(M)$ is trivial, and there is nothing to do. So suppose $M'$ is $\sigma$-compact. If $M$ is locally compact then $(\omom,\le_\infty)\times \omega \gtq \omega \gtq \mathbf{k}(M)$, by (1). So suppose $M$ is not locally compact.
Write $M=\bigcup_n L_n$, an increasing union of compact sets. Fix $n$. Pick a local base for $L_n$ in $M$, say $(B_m^n)_m$ such that $L_n \subseteq B^n_{m+1}\subseteq \cl{B^n_{m+1}} \subseteq B^n_m$.
Let $C^n_m=\cl{B^n_{m}}\setminus B^n_{m+1}$. Note $C^n_m$ is closed.

Define $\phi_+:\omom \times \omega \to \K(M)$ by
\[ \phi_+(f,n) = L_n \cup \bigcup \left\{  L_q \cap C^{n'}_p : p \in \omega, n' \le n \, \& \, q \le \max \{f(0),\ldots, f(p)\} \right\}.\]
Note that $\phi_+(f,n)$ is compact.
Define $ \psi_-:\mathop{CS}(M) \to \omom \times \omega$ as follows. Take any $S$ from $\mathop{CS}(M)$. Then $S$ converges to some point $x_S$, and fix $n_S$ to be the first $n$ such that $x_S$ is in $L_n$. If $S \cap L_{n_S}$ is infinite then set $ \psi_-(S)=(f,n_S)$ where $f$ is arbitrary. Otherwise, note that $S \cap C_m^{n_S}$ is finite for all $m$, and non-empty for infinitely many $m$. For each $m$ let $m^+$ be the first element of $\omega$ greater than or equal to $m$ such that  $S \cap C_{m^+}^{n_S} \ne \emptyset$.
Define $f_S(m)$ to be least such that $S \cap C_{m^+}^{n_S}$ is contained in $L_{f_S(m)}$.
Set $ \psi_-(S)=(f_S,n_S)$.

Suppose $ \psi_-(S) (\le_\infty \times \le) (g,t)$.
We need to show $S \nci \phi_+(g,t)$, or equivalently, $S \cap \phi_+(g,t)$ is infinite.
Write $ \psi_-(S)=(f_S,n_S)$ as in the definition.
We know $n_S \le t$, and that there is an infinite subset $A_S$ of $\omega$ such that $f_S(m) \le g(m)$ for all $m$ in $A_S$.
As in the definition of $ \psi_-(S)$ there are two cases. If $S \cap L_{n_S}$ is infinite then $S \cap \phi_+(g,t) \supseteq S \cap L_t \supseteq S \cap L_{n_S}$, and we are done in this case.
Otherwise, for each $m$ in $A_S$, $S \cap C_{m^+}^{n_S} \ne \emptyset$. So, for each $m$ in $A_S$, $\emptyset \ne S \cap C_{m^+}^{n_S} \subseteq L_{f_S(m)} \cap C_{m^+}^{n_S}  \subseteq \phi_+(g,t)$, taking $p=m^+\ge m$, $n'=n_S \le t$ and $q=f_S(m) \le g(m) \le \max \{g(0),\ldots,g(m^+)\}$ in the definition of $\phi_+(g,t)$. So $S \cap \phi_+(g,t)$ is indeed infinite in this case.
  limit

\medskip

\noindent {\textbf{For (3):}}
Towards the forward direction, suppose either $M'$ is not ($\sigma$-compact or empty) or $M^\#$ is not (compact or empty). In the first case, from (4) we know $(\omom,\le_\infty)\times \omega \not\gtq \mathbf{k}(M)$. Since $(\omom,\le_\infty)\times \omega \gtq (\omom,\le_\infty)\with \omega$, we see $(\omom,\le_\infty)\with \omega \not\gtq \mathbf{k}(M)$.
 Otherwise from Theorem~\ref{th:lb_k}(4) (the reverse direction, for which a self-contained argument is given above) we know $\mathbf{k}(M) \gtq (\omom,\le_\infty)\times\omega$. And recall, $(\omom,\le_\infty)\times \omega \gtq (\omom,\le_\infty)\with \omega$.

Finally, suppose $M^\#$ is compact. We define $\phi_+:\omom \times \omega \to \K(M)$ and $ \psi_-:\mathop{CS}(M) \to \omom \oplus \omega$ and verify $( \psi_-,\phi_+)$ is a morphism from $(\omom,\le_\infty) \with \omega$ to $\mathbf{k}(M)$.
Write $\mathop{LC}(M)=\bigcup_n U_n$ as a union of open sets with compact closure, where $\cl{U_n} \subseteq U_{n+1}$.
Write $M=\bigcup_n L_n$, where $L_n=M^\# \cup \overline{U_n}$. Pick a local base for $M^\#$ in $M$, say $(B_m)_m$ such that $B_{m+1}\subseteq \cl{B_{m+1}} \subseteq B_m$.
Let $C_m=\cl{B_{m}}\setminus B_{m+1}$. Note $C_m$ is closed.

Define $\phi_+:\omom \times \omega \to \K(M)$ by
\[ \phi_+(f,n) = L_n \cup \bigcup \left\{  L_q \cap C_p : p \in \omega \ \& \ q \le \max \{f(0),\ldots, f(p)\} \right\}.\]
Note that $\phi_+(f,n)$ is compact.
Define $ \psi_-:\mathop{CS}(M) \to \omom \oplus \omega$ as follows. Take any $S$ from $\mathop{CS}(M)$. Then $S$ converges to some point $x_S$. If $x_s \notin M^\#$ then fix $ \psi_-(S)$ to be the first $n$ such that $x_S$ is in $U_n$.
Suppose $x_S$ is in $M^\#$.
If $S \cap M^\#$ is infinite then set $ \psi_-(S)=f_S$ where $f_S$ is arbitrary. Otherwise, note that $S \cap C_m$ is finite for all $m$, and non-empty for infinitely many $m$. For each $m$ let $m^+$ be the first element of $\omega$ greater than or equal to $m$ such that  $S \cap C_{m^+} \ne \emptyset$.
Define $f_S(m)$ to be least such that $S \cap C_{m^+}$ is contained in $L_{f_S(m)}$.
Set $ \psi_-(S)=f_S$.

Suppose $ \psi_-(S) (\le_\infty \with \le) (g,t)$.
We need to show $S \nci \phi_+(g,t)$, or equivalently, $S \cap \phi_+(g,t)$ is infinite.
There are three cases. In the first $x_S$ is not in $M^\#$, and so $n= \psi_-(S) \le t$ is in $\omega$, and $x_S \in U_n$. As $S$ converges to $x_S$, all but finitely many terms of $S$ are in $U_n \subseteq L_n  \subseteq L_t \subseteq \phi_+(g,t)$, and we are done in this case.
In the second and third cases $x_S$ is in $M^\#$, and so $f_S= \psi_-(S)$ is in $\omom$.
In the second case $S \cap M^\#$ is infinite, and $S \cap M^\# \subseteq M^\# \subseteq L_t \subseteq \phi_+(g,t)$, and done.
In the final case, from $f_S \le_\infty g$
there is an infinite subset $A_S$ of $\omega$ such that $f_S(m) \le g(m)$ for all $m$ in $A_S$.
From the definition of $ \psi_-(S)$ in this case,  for each $m$ in $A_S$, $S \cap C_{m^+} \ne \emptyset$. So, for each $m$ in $A_S$, $\emptyset \ne S \cap C_{m^+} \subseteq L_{f_S(m)} \cap C_{m^+}  \subseteq \phi_+(g,t)$, taking $p=m^+\ge m$ and $q=f_S(m) \le g(m) \le \max \{g(0),\ldots,g(m^+)\}$ in the definition of $\phi_+(g,t)$. So $S \cap \phi_+(g,t)$ is indeed infinite in this case.
\end{proof}

\begin{thm}\label{th:analytic} Let $M$ be separable metrizable.

(1) The following are equivalent:

\quad (i) $M$ is analytic, (ii)  $\omom \gtq \mathbf{kc}(M)$, and (iii) $\omom \gtq \mathbf{k}(M)$.

(2)    If $M$ is analytic and not $\sigma$-compact then $\mathbf{k}(M) \gte \omom \gte \mathbf{kc}(M)$.
\end{thm}
\begin{proof}
We show claim (1). The equivalence of (i) and (ii) is due to Christensen. Since $\mathbf{k}(M) \gtq \mathbf{kc}(M)$, except for $M$ with $M'$ compact,  (iii) immediately implies (ii). If (ii) holds, and $\omom \gtq \mathbf{kc}(M)$  then Theorem~\ref{th:kc^om_above_k} says that $(\omega^\omega)^\omega \gtq \mathbf{k}(M)$. Since $(\omega^\omega)^\omega \te \omega^\omega$, (iii) follows.

For claim (2), suppose $M$ is analytic but not $\sigma$-compact, then (a) $M'$ is not compact, and so by Lemma~\ref{l:above_kc} we have $\mathbf{k}(M)\gtq \mathbf{kc}(M)$, and  (b) we know from \cite{FG-ShapeCptCovers} that  $\mathbf{kc}(M) \gtq \omom$. Apply claim (1) to complete the argument.
\end{proof}

A space $X$ is \emph{Menger} if for every sequenece, $(\mathcal{U}_n)_n$, of open covers of $X$ we can find, for each $n$, finite subcollections, $\mathcal{V}_n$ of $\mathcal{U}_n$, such that $\bigcup_n \mathcal{V}_n$ covers $X$. Every $\sigma$-compact space is Menger, but there exist, in ZFC, separable metrizable Menger spaces which are not $\sigma$-compact.
If $M$ is separable metrizable and $|M| < \mathfrak{d}$ then it is Menger.

\begin{prop}\label{pr:menger} \

(1) If $M$ is not Menger then $\mathbf{k}(M) \gtq \omom$.

(2) Assuming $\mathfrak{b} < \mathfrak{d}$,  every separable metrizable space, $M$, with $|M|=\omega_1$ has $\mathbf{K}(M)$, in the Tukey order, strictly above $(\omom,\le_\infty) \times \omega$ but incomparable with $\omom$.
\end{prop}
\begin{proof}
     For (1) we recall from \cite{FG-ShapeCptCovers} that a separable metrizable space $M$ is Menger if and only if $\mathbf{kc}(M) \not\gtq \omom$, and apply Theorem~\ref{th:analytic} and Lemma~\ref{l:above_kc}.
    Hence if $M$ is not Menger then $\mathbf{k}(M) \gtq \mathbf{kc}(M) \gtq \omom$.

    Now, for (2), suppose $\mathfrak{b} < \mathfrak{d}$, and  $M$ is separable metrizable of size $\omega_1$. As $\omega_1 < \mathfrak{c}$ we see $M$ is not $\sigma$-compact (all its compact subsets are countable), and so $\mathbf{k}(M) >_T (\omom,\le_\infty) \times \omega$.
We show now that $\mathop{seq}(M)$, the cofinality of $\mathbf{seq}(M)$ is no more than $\mathfrak{b}$.
Since $\mathfrak{d} > \mathfrak{b} \ge \mathop{seq}(M) \ge k(M)$,
the cofinality of $\mathbf{k}(M)$;  for this $M$, we have $\mathbf{k}(M) \not\gtq \omom$.

  By Theorem~\ref{th:seq_rep}, as $M$ has size $\omega_1$, we know $P(\omega_1) \times (\omega_1,=) \tq \mathbf{seq}(M)$. So it suffices to show $P(\omega_1)$ has cofinality no more than $\mathfrak{b}$. To see this, first note that $P(\omega_1)=((([\omega_1]^{<\omega})^\omega)_\infty, i_\infty)$ is the union over all infinite $\alpha$ in $\omega_1$ of $((([\alpha]^{<\omega})^\omega)_\infty, i_\infty)$. Now observe that such $((([\alpha]^{<\omega})^\omega)_\infty, i_\infty)$ is isomorphic to $((([\omega]^{<\omega})^\omega)_\infty, i_\infty)$, which is, by Lemma~\ref{l:small_Pk}, Tukey equivalent to $(\omom,\le_\infty)$, and the latter has cofinality $\mathfrak{b}$.
\end{proof}

Recall: $\mathbf{0}=(\emptyset,\emptyset,\emptyset)$ (the relation with empty domain), $\mathbf{1}=(1,=)$, $\omega$ is the first infinite ordinal, $\omom$ has the product order, $f \le_\infty g$ if for infinitely many $n$, $f(n) \le g(n)$; $\with$ is the Tukey least upper bound, and $\times$ is the standard product. Recall: $M'$ is the set of non-isolated points of $M$, while $M^\#$ is the set of points of non-local compactness of $M$.

\begin{thm}[Initial Structure] \label{th:k_init} Below $M$ is separable metrizable. Under the  Tukey order the initial structure of $\mathbf{k}(\mathcal{M})$ is as follows.

\medskip

(0) $\mathbf{k}(M) \gte \mathbf{0}$ if and only if $M$ is discrete, or equivalently $M'=\emptyset$. This is the minimum  Tukey class.

\medskip

(1) Its immediate successor is the class of $\mathbf{1}$, and $\mathbf{k}(M) \gte \mathbf{1}$ if and only if $M'$ is compact, or equivalently, $M=K \oplus D$ where $K$ is compact and $D$ is discrete.

This class has two immediate successors, which are incomparable.

\medskip

(2a) $\omega$, where $\mathbf{k}(M) \gte \omega$ if and only if $M'$ is not compact but $M$ is locally compact (i.e. $M^\#=\emptyset$).

\medskip

(2b) $(\omom,\le_\infty)$, where $\mathbf{k}(M) \gte (\omom,\le_\infty)$ if and only if $M'$ is compact but $M^\# \ne \emptyset$ (i.e. $M$ not locally compact).

There is a  unique immediate successor of the previous two classes.

\medskip

(3) $(\omom,\le_\infty) \with \omega$, where $\mathbf{k}(M) \gte (\omom,\le_\infty) \with \omega$ if and only if $M^\#$ compact (this implies $M'$ $\sigma$-compact).

This class has a unique immediate successor.

\medskip

(4) $(\omom,\le_\infty) \times \omega$, where $\mathbf{k}(M) \gte (\omom,\le_\infty) \times \omega$ if and only if $M'$ is $\sigma$-compact but $M^\#$ is not compact.

\smallskip

The classes to this point exhaust all $M$ that are the disjoint sum of $\sigma$-compact and discrete. Class (4) has an immediate successor.

\medskip

(5) $\omom$, where $\mathbf{k}(M) \gte \omom$ if and only if $M$ is analytic but not $\sigma$-compact.

\smallskip

However, this is, consistently at least, not unique.

\medskip

(6) If $M$ is not Menger then $\mathbf{k}(M) \gtq \omom$. Consistently there are
 Menger sets $M$  such that $\mathbf{k}(M)$ lies strictly above  $(\omom,\le_\infty) \times \omega$ in the  Tukey order, but is incomparable with $\omom$.
\end{thm}
\begin{proof}
    Claims (0) through (4) follow from Theorems~\ref{th:lb_k} and \ref{th:ub_k}. Claim (5) from  Theorem~\ref{th:analytic}.
   The fact that class (5) is an immediate successor of (4) then follows from the result that analytic Menger sets are $\sigma$-compact.
Claim (6) follows from Proposition~\ref{pr:menger}.
  \end{proof}

As examples illustrating the eight types, consider: (0) the one point space, $M=\{1\}$, with $\mathbf{k}(M) \gte \mathbf{0}$, (1) the convergent sequence, $M=\omega+1$, with $\mathbf{k}(M) \gte \mathbf{1}$, (2a) a countable sum of convergent sequences, $M=\omega \times (\omega+1)$, where $\mathbf{k}(M) \gte \omega$, (2b) the metric fan, with $\mathbf{k}(\mathbb{F}) \gte (\omom,\le_\infty)$, (3) $M=\mathbb{F} \times (\omega+1)$, where $\mathbf{k}(M) \gte (\omom,\le_\infty)\with \omega$, (4) the square of the metric fan, where $\mathbf{k}(\mathbb{F}^2) \gte (\omom,\le_\infty)\times \omega$, and (5) $M$ the irrationals, $\mathbf{k}(M) \gte \omom$.  For  (6) see Proposition~\ref{pr:menger}.

Recall that the property of being a $k$-space is not in general productive.
Relatedly, the examples mentioned above show that the behavior of products in relation to $\mathbf{k}(M)$ is intriguing: squaring a space may lead to a strictly more complex $k$-structure, while multiplication by a compact space may strictly increase the complexity of the $k$-structure, and multiplication by different compact spaces may result in different outcomes with respect to $k$-structure.

\begin{ex} We have:
$\mathbf{k}(\mathbb{F}\times I)\gte \mathbf{k}(\mathbb{F}^2) >_{T} \mathbf{k}(\mathbb{F}\times (\omega+1))>_{T} \mathbf{k}(\mathbb{F})$.
\end{ex}

\subsection{The Cofinal Structure of $\mathbf{k}(\mathcal{M})$ and Counting Types}

Recalling that (Lemmas~\ref{l:below_K} and \ref{l:above_kc}) for metrizable $M$ with $M'$ not compact we have $\K(M) \gtq \mathbf{k}(M) \gtq \mathbf{kc}(M)\gte (M,\K(M),\in)$, from the results 3.11 through 3.14 of  \cite{GM1} we deduce  there is a wide variety of $k$-structures of separable metrizable spaces, with complex behavior under the  Tukey order.
\begin{thm} \

(1) There are $2^\mathfrak{c}$-many  Tukey types of $\mathbf{k}(M)$, for $M$ separable metrizable.

(2)     There is a $2^\mathfrak{c}$-sized family $\mathcal{M}$ of separable metrizable spaces such that if $M$ and $N$ are distinct elements of $\mathcal{M}$ then  $\mathbf{k}(M) \not\gtq \mathbf{k}(N)$.

(3) $\mathfrak{c}^+$, $I$ and $(\mathbb{P}(\omega),\subseteq)$ all order embed in $\mathbf{k}(\mathcal{M})$ under the  Tukey order.
\end{thm}
Consistently at least,  we can find a large family of Menger sets with distinct $k$-structures up to  Tukey equivalence, none with $k$-structure $\gtq \omom$.
\begin{thm}\label{th:many_menger}
    If $\aleph_1=\mathfrak{b}$, $\mathfrak{d}=\aleph_2=\mathfrak{c}$ and $2^\mathfrak{b} > \mathfrak{c}$ then there is a $2^\mathfrak{b}$-sized family $\mathcal{S}$ of Menger subsets of the reals such that

    (i) for all $M$ from $\mathcal{S}$ we have $\mathbf{k}(M)$ incomparable with $\omom$ and strictly above $(\omom,\le_\infty) \times \omega$, and

    (ii) $\mathbf{k}(M) \not\gte \mathbf{k}(N)$ for every distinct $M$ and $N$ from $\mathcal{S}$.
\end{thm}
\begin{proof}
    As shown in \cite{FG-ShapeCptCovers},  there is, in \textsc{ZFC},  a $2^\mathfrak{b}$-sized family $\mathcal{S}'$ of Menger subsets of the reals all of size $\mathfrak{b}$. By Proposition~\ref{pr:menger}, all of these have $\mathbf{k}$ incomparable with $\omom$ and strictly above $(\omom,\le_\infty) \times \omega$.
As shown in \cite{GM1} (see Corollary~4.2(4)) every instance of a Tukey quotient $\mathcal{K}(N) \tq (M,\mathcal{K}(M))$, where $M$ and $N$ are separable metrizable is `realized' by a closed subset of the Hilbert cube. It follows that given $N$ there are only $\mathfrak{c}$-many $M$ with $\mathcal{K}(N) \tq (M,\mathcal{K}(M))$. So if we define an equivalence relation $\sim$ on $\mathcal{S}$ by $M \sim N$ if and only if $\mathcal{K}(N) \tq (M,\mathcal{K}(M))$ and $\mathcal{K}(M) \tq (N,\mathcal{K}(N))$, then the equivalence classes have size no more than the continuum.
Hence when $2^\mathfrak{b} > \mathfrak{c}$ we can extract from $\mathcal{S}'$ a $2^\mathfrak{b}$-sized subcollection, $\mathcal{S}$, where distinct elements of $\mathcal{S}$ are $\sim$-incomparable.
From $\K(M) \gtq \mathbf{k}(M) \gtq  (M,\K(M),\in)$ we deduce that if $M$ and $N$ are distinct elements of $\mathcal{S}$ then $\mathbf{k}(M) \not\gte \mathbf{k}(N)$, as required.
\end{proof}

Let us verify that the  configuration of cardinals in Theorem~\ref{th:many_menger} is consistent. To see this start with a model, $V_0$, of (GCH). Add $\aleph_3$-many Cohen subsets of $\aleph_1$, to get $V_1$. In $V_1$ the Continuum Hypothesis is still true, but $2^{\aleph_1}=\aleph_3$. Now iterate Miller forcing in length $\omega_2$, to get model, $V_2$, of set theory. This model has all the required features.
Briefly, this is because the iteration of Miller forcing is proper. Since conditions in Miller forcings are trees of finite sequences of integers, and $V_0$ satisfies (CH), the Miller iteration has the $\aleph_2$-cc, so all cardinals from $\aleph_2$ are preserved. Each Miller iterand adds a real, so in $V_2$ the continuum is $\aleph_2$.
In fact, because Miller forcing adds an unbounded real, $\mathfrak{d}=\aleph_2$.
But $\mathfrak{b}=\aleph_1$. One way to see this is to recall a standard fact about iterating Miller forcing over (CH) is that the full iteration preserves $P$-points, and so does not add dominating reals.

\section{The Sequential Structure} \label{s:seq}

\subsection{General Results}

We start by recording the natural connection between the sequential and $k$-structures. The proof is left to the reader.
\begin{lem} For any sequential space $X$ we have
    $\mathbf{seq}(X) \gtq \mathbf{k}(X)$.
\end{lem}

The next lemma significantly simplifies the relation, $\mathbf{seq}(X)$.
\begin{lem} Let $X$ be sequential. Then
\[\mathbf{seq}(X) \gte (\mathop{CS}(X),\mathop{CS^+}(X),\neg c) \gte  (\mathop{CS}(X), {ii}),\] where $A {ii} B$ if and only if $A$ has infinite intersection with $B$.
\end{lem}

\begin{proof} Using the identity map for both  $ \psi_-$ and $\phi_+$, it is straightforward to verify that    $\mathbf{seq}(X) \gtq (\mathop{CS}(X),\mathop{CS^+}(X),\nci)$.
In the other direction, again take the identity map for $\phi_+$.  Choose $F$ to be a non-closed subset of $X$, by sequentiality of $X$ there is a $C\in \mathop{CS}^+(X)$  such that $F(\nci)C$, and define $ \psi_-(F)=C^-$ where $C^-$ is the sequence obtained by removing the limit point of $C$. Now $( \psi_-,\phi_+)$ witness  $(\mathop{CS}(X),\mathop{CS^+}(X),\nci)\gtq \mathbf{seq}(X)$.

It remains to show that $(\mathop{CS}(X),\mathop{CS^+}(X),\nci)\gte (\mathop{CS}(X), {ii})$. Once we notice that for $D$  a convergent sequence without limit and $C$  a convergent sequence with limit, $D(\nci)C$ if and only if $D{ii}C$,  this is clear.
\end{proof}

Now we start to deconstruct $\mathbf{seq}(X)$, using the sum operation of relations.
\begin{lem} For any space $X$ we have
    $\mathbf{seq}(\bigoplus_{\lambda \in \Lambda} X_\lambda) \gte \sum_{\lambda \in \Lambda} \mathbf{seq}(X_\lambda)$.
\end{lem}
\begin{proof} First we show that $\sum_{\lambda \in \Lambda} \mathbf{seq}(X_\lambda)\gtq \mathbf{seq}(\bigoplus_{\lambda \in \Lambda} X_\lambda)$.  Let $F$ be a non-closed subset of $\bigoplus_{\lambda \in \Lambda } X_\lambda$. Fix some $\lambda_F$ so that $F\cap X_\lambda$  is non-closed in $X_\lambda$. Set  $ \psi_-(F) = X_{\lambda_F}\cap F$. We use the identity map as $\phi_+$. Then $( \psi_-,\phi_+)$ is the desired morphism.

For the converse, we use identity map as $\psi_-$. For  $S^+$  a convergent sequence in $\bigoplus_{\lambda \in \Lambda} X_\lambda$, let $X_{\lambda_{S^+}}$ be the term of the disjoint sum containing the limit of $\lambda_{S^+}$, and  define $\psi_+(S^+) = S^+\cap X_{\lambda_{S^+}}$, which is clearly a convergent sequence with limit in  $X_{\lambda_C}$. Then $(\psi_-,\psi_+)$ is the desired morphism.
\end{proof}

For any space $X$ and point $x$ in $X$, let $\mathop{CS}(x,X)$ be all $S$ in $\mathop{CS}(X)$ converging to $x$. Let $\mathbf{seq}(x,X)=(\mathop{CS}(x,X),\mathop{CS}(x,X), \nci) \gte (\mathop{CS}(x,X),ii)$.

\begin{lem} Let $X$ be sequential.
Then $\mathbf{seq}(X) \gte \sum_{x \in X} \mathbf{seq}(x,X)$.
\end{lem}
\begin{proof}
 This is clear after unpacking definitions, and noting that $\mathop{CS}(X)$ is the disjoint union over all $x$ in $X$ of $\mathop{CS}(x,X)$. Recall we can take $\mathbf{seq}(X)=(\mathop{CS}(X),ii)$, and $\sum_{x \in X} \mathbf{seq}(x,X)=(\bigoplus_{x \in X} \mathop{CS}(x,X), \bigoplus_{x \in X} \mathop{CS}(x,X), II)$, where $S (II) T$ if and only if $S,T$ are in some $\mathop{CS}(x,X)$ and $S \cap T$ is infinite.
 Now $( \psi_-,\phi_+)$ and $(\psi_-,\psi_+)$ are the required morphisms if we take $\phi_+$ and $\psi_-$ to be the map taking $S$ in $\mathop{CS}(X)$ to $S$ in $\mathop{CS}(x,X)$, where $x$ is the limit of $S$, and $ \psi_-,\psi_+$ on each $\mathop{CS}(x,X)$ to be the inclusion map into $\mathop{CS}(X)$.
\end{proof}

\begin{lem}\label{l:inj}
    Suppose $\sum_{x \in X} \mathbf{seq}(x,X)\gtq \sum_{y \in Y} \mathbf{seq}(y,Y)$ via the morphism $( \psi_-,\phi_+)$.
    Suppose for $i=0,1$ we have  $S_i$ in $\mathbf{seq}(y_i,Y)$  and $ \psi_-(S_i)$ in $\mathbf{seq}(x_i,X)$.  If $y_0 \ne y_1$ then $x_0 \ne x_1$. Hence
         there is an injection of $Y'$ into $X'$.
\end{lem}
\begin{proof}
    As $y_0 \ne y_1$ we see $S_0$ and $S_1$ do not have an upper bound (such would have to converge to both $y_0$ and $y_1$). So $ \psi_-(S_0)$ and $ \psi_-(S_1)$ do not have an upper bound. Since any two elements of a $\mathbf{seq}(x,X)$ do have an upper bound (their union, for example) we see $x_0$ and $x_1$ must be different.

    Recall $\mathbf{seq}(y,Y)$ is empty if and only if $y$ is isolated. So for any $y$ in $Y'$, pick some $S$ in $\mathbf{seq}(y,Y)$ and let $g(x)$ be such that $ \psi_-(S)$ is in $\mathbf{seq}(g(x),X)$. As $\mathbf{seq}(g(x),X)$ is non-empty we have that $g(x)$ is in $X'$, and the first part says $g$ is an injection.
\end{proof}

Next we uncover the structure of $\mathbf{seq}(X)$ of a first countable space $X$, connecting it to the special relation $\left(\left(\prod_n [\kappa_n]^{<\omega}\right)_\infty,i_\infty \right)$ from Section~\ref{ss:ssr}.
\begin{lem}
    Suppose $x$ in $X$ has a local base $(B_n)_n$ which is strictly decreasing, $B_n \supsetneq B_{n+1}$. Let $D_n=B_n \setminus B_{n+1}$ and $\kappa_n=|D_n|$.
    Then $\mathbf{seq}(x,X) \gte  \left(\left(\prod_n [\kappa_n]^{<\omega}\right)_\infty,i_\infty \right)$.
\end{lem}
\begin{proof}
    Recall $\mathbf{seq}(x,X)=(\mathop{CS}(X),ii)$. Note that $\left(\left(\prod_n [\kappa_n]^{<\omega}\right)_\infty,i_\infty \right)$ is clearly  Tukey equivalent to $\left(\left(\prod_n [D_n]^{<\omega}\right)_\infty,i_\infty \right)$.

    Define $\phi_+:\mathop{CS}(x,X) \to \prod_n [\kappa_n]^{<\omega}$ by $\phi_+(S)=(S\cap D_n)_n$, and $ \psi_-:\prod_n [\kappa_n]^{<\omega} \to \mathop{CS}(x,X)$ by $ \psi_-((F_n)_n)=\bigcup_n F_n$. Note they are both well-defined (each $S\cap D_n$ is a finite subset of $G_n$ and infinitely many are $\ne \emptyset$, and $\bigcup_n F_n$ is an infinite sequence converging to $x$).
    Suppose $ \psi_-((F_n)_n) ii S$. Then $\bigcup_n F_n \cap S$ is infinite. Let $A=\{n \in \omega : F_n \cap S \ne \emptyset\}$. As each $F_n \cap S$ is finite, $A$ is infinite, and for each $n$ from $A$, $\emptyset \ne F_n \cap S = F_n \cap (S \cap D_n)$ (as $F_n \subseteq D_n$), as required for $(F_n)_n i_\infty \phi_+(S)$.

    Define $\psi_+((F_n)_n)=\bigcup_n F_n$ and $\psi_-(S)=(S \cap D_n)_n$. Now if $\psi_-(S) i_\infty (G_n)_n$ then there is an infinite $A$ such that for every $n$ from $A$, $\emptyset \ne  G_n \cap (S \cap D_n)=G_n \cap S $. Since the $G_n \cap S$ are disjoint, and non-empty for $n$ in $A$, we see $S \cap \bigcup_n G_n = \bigcup_n (S \cap G_n)$ is infinite, and so $S ii \phi_+((G_n)_n)$, as required.
\end{proof}

Observe, if $x$ be a point of first countability in a  space $X$ then  $x$ is isolated if and only if $\mathbf{seq}(x,X) \gte \mathbf{0}$, and  $\mathbf{seq}(x,X) \gte \mathbf{1}$ if and only if $x$ is non-isolated and has a neighborhood homeomorphic to $\omega+1$, or equivalently, has a strictly decreasing local base, $(B_n)_n$, where the differences, $D_n=B_n \setminus B_{n+1}$ are all finite. We continue to deconstruct the possible types of $\mathbf{seq}(x,X)$ of a point $x$ in a space, $X$, based on the size of its neighborhoods.

For a space $X$ and $x$ in $X$ fix $U_x$ an open neighborhood of $x$ of minimal cardinality. Note that every neighborhood of $x$ contained in $U_x$ has cardinality $|U_x|$.
Let $\mathcal{S}_X=\{|U_x| : x \in X\}$.
 Let $X_\kappa$ (respectively, $X_{\le \kappa}$ and $X_{<\kappa}$) be the set of $x$ in $X$ with $|U_x|$ equal to $\kappa$ (respectively, $\le \kappa$ or $<\kappa$).
In the above notation the following is clear.
\begin{lem}
    The sets $X_{\le \kappa}$ and $X_{<\kappa}$ are open, and $X=\bigcup_{\kappa \in \mathcal{S}_X} X_{\le \kappa} = \bigcup_{\kappa \in \mathcal{S}_X} X_{\kappa}$.
\end{lem}

We now focus on the case when our space, $M$ say, is separable and metrizable.
\begin{lem} Let $M$ be separable and metrizable. Then:
(1) $\mathcal{S}_M$ is countable,
(2) if $\kappa \in \mathcal{S}_M$ then $\kappa \le |M_{\le \kappa}| \le \kappa . \aleph_0$, and
(3) $\sup \mathcal{S}_M = |M|$.
\end{lem}
\begin{proof}
Claim (1) is immediate once we note that we can take all minimal neighborhoods from a countable base for $M$.

For (2), as $\kappa$ is in $\mathcal{S}_M$, there is a point $x$ in $M_\kappa$ with minimal neighborhood $U_x$ where $|U_x|=\kappa$. Now every one of the $\kappa$-many points in $U_x$ are in $M_{\le \kappa}$. Hence $M_{\le \kappa}$ has at least $\kappa$-many elements. On the other hand, using minimal open sets from a countable base, we see $M_{\le \kappa}$ is covered by countably many (open) sets all of cardinality $\le \kappa$. So $|M_{\le \kappa}| \le \kappa . \aleph_0$.

Now (3) follows from (2) and $M=\bigcup_{\kappa \in \mathcal{S}_M} M_{\le \kappa}$.
\end{proof}

Separate the points of $M_\kappa$ in two as follows. Let $M_\kappa^+$ be all $x$ from $M_\kappa$ which have a neighborhood base $(B_n)_n$ such that every $D_n=B_n \setminus B_{n+1}$ has cardinality $\kappa$. Let $M_\kappa^-$ be $M_\kappa \setminus M_\kappa^+$.
Note that a point $x$ from $M_\kappa$ is in $M_\kappa^-$ if and only if it has a neighborhood base $(B_n)_n$ such that every $D_n=B_n \setminus B_{n+1}$ has cardinality $<\kappa$, and further observe we can suppose the $|D_n|$'s increase strictly up to $\kappa$, which must have countable cofinality.

\begin{lem} Let $M$ be separable metrizable.

(1)    If $\cof(\kappa) \ne \omega$ then $M_\kappa^-=\emptyset$ (so $M_\kappa=M_\kappa^+$) and $|M_\kappa|=\kappa.\aleph_0$.

(2) If $M_\kappa^- \ne \emptyset$ then  $M_\kappa^-$ is discrete, and so $|M_\kappa^-|$ is in $\omega+1$.
\end{lem}
\begin{proof}
    We have just observed that if $M_\kappa^- \ne \emptyset$ then $\kappa$ must have countable cofinality. So if $\cof(\kappa) \ne \omega$ then indeed $M_\kappa^-=\emptyset$. Now we know $|M_{\le \kappa}| \kappa$ , and we see $M_{<\kappa}$ is a countable union of sets (the $M_{\le \lambda}$ for $\lambda \in \mathcal{S}_M\setminus\{\kappa\}$) all of cardinality $<\kappa$, and thus $|M_\kappa|=\kappa$.

    Now suppose, for a contradiction, $M_\kappa^- \ne \emptyset$ but is not discrete, say $x$ is in  $M_\kappa^-$ and is the limit of the sequence $S$ also in $M_\kappa^-$.
    Pick a decreasing neighborhood base $(B_n)_n$ of $x$ so that $B_1=U_x$ and, for each $n$,  $S$ meets the interior of $D_n=B_n \setminus B_{n+1}$.
    As $x$ is not in $M_\kappa^+$, we can find an $n$ and $x'$ in $S$ so that $|D_n|<\kappa$ and $x'$ is in the interior of $D_n$. But $x'$ is in $M_\kappa$, so has no neighborhood of size $<\kappa$, contradiction.
\end{proof}

\subsection{$\mathbf{seq}(\mathcal{M})$ - Representation, Realization, and Structure \& Number of Types}
By $\mathbf{seq}(\mathcal{M})$ we mean the  Tukey types (equivalence classes) of $\mathbf{seq}(M)$, where $M$ is separable metrizable, ordered by $\gtq$.

Recall, from Section~\ref{ss:ssr}, the relations  $P(\kappa)=((([\kappa]^{<\omega})^\omega)_\infty, i_\infty)$,  and $Q(\kappa)=\mathbf{0}=S(\kappa)$ if $\cof(\kappa) \ne \omega$, and
otherwise select a sequence, $(\kappa_n)_n$, strictly increasing up to $\kappa$, and then $Q(\kappa)=(( \prod_n [\kappa_n]^{<\omega}    )_\infty,i_\infty)$, and $S(\kappa)=\sum_n P(\kappa_n) \times (\kappa_n,=)$.

\begin{thm}[Representation]\label{th:seq_rep} Let $M$ be separable metrizable, $\kappa=|M|$.

Then $\mathbf{seq}(M) \gte P(\kappa) \times (\lambda_M,=) + Q(\kappa) \times (\mu_M,=) + S(\kappa)$,
where $\lambda_M=|M_\kappa^+|$ is  $\le \kappa$ and $\mu_M=|M_\kappa^-|$ is in $\omega+1$.
\end{thm}
\begin{proof}
Let  $P=\sum_{x \in M_\kappa^+} \mathbf{seq}(x,M)$,  $Q=\sum_{x \in M_\kappa^-} \mathbf{seq}(x,M)$ and $S=\sum_{x \in M_{<\kappa}} \mathbf{seq}(x,M)$. Then $\mathbf{seq}(M)=P+Q+S$, and $P\gte P(\kappa) \times (\lambda_M,=)$ and $Q \gte Q(\kappa) \times (\mu_M,=)$.

We start with the case when $\lambda_M=\kappa$. Then both $\mathbf{seq}(M) = P+Q+S$ and $P(\kappa) \times (\lambda_M,=) + Q(\kappa) \times (\mu_M,=) + S(\kappa)$ are $\gtq P=P(\kappa) \times (\kappa,=)$.
So it suffices to show $P(\kappa) \times (\kappa,=) \gtq P+Q+S$ and $P(\kappa) \times (\kappa,=) \gtq  P(\kappa) \times (\lambda_M,=) + Q(\kappa) \times (\mu_M,=) + S(\kappa)$.
But there are no more than $\kappa$ points of $M$ not in $M_\kappa^+$, and each such point $x$ satisfies $P(\kappa) \gtq \mathbf{seq}(x,M)$. Hence, working term-by-term, $P(\kappa)\times (\kappa,=) \gtq Q+S$. Similarly, $P(\kappa)\times (\kappa,=) \gtq Q(\kappa) \times (\mu_M,=) + S(\kappa)$.
Since $P(\kappa) \times (\kappa,=) \gte P(\kappa) \times (\kappa,=)+P(\kappa)\times (\kappa,=)$ the claim follows.

Now suppose $\lambda_M<\kappa$. Then $\kappa$ has countable cofinality.
We show $S \gte S(\kappa)$, which completes the proof.

Let $\mathcal{S}_M'=\mathcal{S}_M \setminus \{\kappa\}$, and note $\sup \mathcal{S}_M'=\kappa$. Enumerate $\mathcal{S}_M'$ as $\lambda_1, \lambda_2, \ldots$. Pick a strictly increasing sequence of cardinals $(\kappa_n)_n$ such that $\kappa_n < \kappa$, $\kappa_n \ge \lambda_1,\ldots,\lambda_n$ and $\kappa_n \ge |M_{\lambda_1}|, \ldots, |M_{\lambda_n}|$. Note $\lim_n \kappa_n=\kappa$.
Now, using this particular sequence in the definition of $S(\kappa)$, we see $S(\kappa) = \sum_n P(\kappa_n) \times (\kappa_n,=) \gtq \sum_{\lambda \in \mathcal{S}_M'} \sum \{\mathbf{seq}(x,M) : x \in M_{\lambda}\}=S$, as required.

For the reverse, first note that members of $\mathcal{S}_M'$ with countable cofinality are limits of elements of $\mathcal{S}_M'$ which do not have countable cofinality. So pick a strictly increasing sequence $(\kappa_n)_n$ inside $\mathcal{S}_M'$ converging to $\kappa$.
Since each $\kappa_n$ does not have countable cofinality we know $|M_{\kappa_n}|=\kappa_n$ and for every $x$ from $M_{\kappa_n}$ we have $\mathbf{seq}(x,M)\gte P(\kappa_n)$.
Thus clearly $S \gtq \sum_n \sum \{\mathbf{seq}(x,M) : x \in M_{\kappa_n}\} = \sum_n P(\kappa_n) \times (\kappa_n,=) = S(\kappa)$, taking this particular sequence in the definition of $S(\kappa)$.
\end{proof}

\begin{thm}[Realization]\label{th:seq_real} \

For every cardinal $\kappa \le \mathfrak{c}$ there is a separable metrizable space $M=M(\kappa)$ such that $\mathbf{seq}(M) \gte P(\kappa) \times (\kappa,=)$.

    For every cardinal $\kappa \le \mathfrak{c}$ of countable cofinality and $\lambda < \kappa$ and $\mu \in \omega+1$ there is a separable metrizable space $N=N(\kappa,\lambda,\mu)$ such that $\mathbf{seq}(N) \gte P(\kappa) \times (\lambda,=) + Q(\kappa) \times (\mu,=) + S(\kappa)$.
\end{thm}
\begin{proof} Our examples of separable metrizable spaces below are all zero dimensional, and so homeomorphic to subspaces of the Cantor set, $\mathbb{C}$. Start by taking any $\kappa \le \mathbf{c}$. Fix a countable base $\mathcal{B}$ for $\mathbb{C}$. For each $B$ from $\mathcal{B}$ select $M_B$ from $[B]^\kappa$. Let $M(\kappa)=M=\bigcup_{B \in \mathcal{B}} M_B$. Then $|M|=\kappa$ and every point of $M$ has a strictly decreasing neighborhood base, $(B_n)_n$, all of whose differences, $B_n \setminus B_{n+1}$ have size $\kappa$. So $\mathbf{seq}(M)=\sum_{x \in M} \mathbf{seq}(x,M)=P(\kappa) \times (\kappa,=)$, as required.
For use below, note that $M(\kappa)$ is dense in the Cantor set.

Now suppose $\kappa \le \mathfrak{c}$ has countable cofinality, and fix a strictly increasing sequence, $(\kappa_n)_n$, of regular cardinals with limit $\kappa$. Also suppose we have $\lambda < \kappa$ and $\mu$ in $\omega+1$.
Let $Q_1$ be $\bigoplus_n M(\kappa_n) \cup \{\ast\}$, where the $N$th basic neighborhood of $\ast$ is $\bigoplus_{n\ge N} M(\kappa_n) \cup \{\ast\}$. Let $Q$ be the disjoint sum of $\mu$-many copies of $Q_1$. Note $Q(\kappa) + S(\kappa) \gtq \mathbf{seq}(Q) \gtq Q(\kappa)$.
Let $P$ be a subset of $M(\kappa)$ of size $\lambda$. Let $T$ be the subspace of $\mathbb{C} \times (\{0\} \cup \{1/n : n \in \mathbb{N}\})$ given by $T=(P \times \{0\}) \cup \bigcup_n (M(\kappa_n) \times \{1/n\})$.
Note, because each $M(\kappa_n)$ is dense in $\mathbb{C}$, and taking, for example, standard $1/n$-balls as basic neighborhoods, that each point in $P \times \{0\}$ has a neighborhood base with differences all equal to $\kappa$. Then $|T|=\kappa$ and $\mathbf{seq}(T) \gte P(\kappa) \times (\lambda,=)  + S(\kappa)$.
Let $N(\kappa,\lambda,\mu)=N = T \oplus Q$.  Then $|N|=\kappa$ and
it follows from the  Tukey calculations above that $\mathbf{seq}(N)=P(\kappa) \times (\lambda,=) + Q(\kappa) \times (\mu,=) + S(\kappa)$, as desired.
\end{proof}





Let $\Delta=\{(\lambda,\mu) : 0 \le \lambda < \omega$ and $ \mu \le \omega\}$, partially ordered by $\succeq$ where $(\lambda,\mu) \succeq (\lambda',\mu')$ if and only if $\lambda \ge \lambda'$ amd $\lambda+\mu \ge \lambda'+\mu'$.
The next theorem says that the  Tukey classes of $\mathbf{seq}(M)$, where $M$ is separable metrizable are ranked by $|M|=\aleph_\alpha$, there is a unique class if $|M|$ does not have countable cofinality, and otherwise is parametrized by $\Delta$ followed by $\alpha+1$.

\begin{thm}[ Tukey Structure]\label{th:seq_gt}
Let $\mathbf{seq}_\kappa(\mathcal{M})$ be all  Tukey classes of $\mathbf{seq}(M)$ where $M$ is separable metrizable  of infinite cardinality $\kappa$.

\smallskip

If  $\mathbf{seq}(M)$ is in $\mathbf{seq}_\mu(\mathcal{M})$ and $\mathbf{seq}(N)$ is in $\mathbf{seq}_\nu(\mathcal{M})$, where $\mu < \nu$, then $\mathbf{seq}(M) <_{T} \mathbf{seq}(N)$.

\medskip

\noindent \textbf{Case $\kappa=\aleph_0$:} \ The  Tukey classes in $\mathbf{seq}_\kappa(\mathcal{M})$ are parametrized by $\Delta+1$, say $\mathbf{seq}(M(\lambda,\mu))$, for $(\lambda,\mu)$ in $\Delta$, and an additional class $\mathbf{seq}(N)$. Here $\mathbf{seq}(M(\lambda,\mu)) \gte (\omom,\le_\infty) \times (\lambda,=) + (\mu,=)$,  $\mathbf{seq}(N)\gte (\omom,\le_\infty) \times (\omega,=)$ is an upper bound of all $\mathbf{seq}(M(\lambda,\mu))$, and $(\lambda,\mu) \succeq (\lambda',\mu')$ if and only if $\mathbf{seq}(M(\lambda,\mu)) \gtq \mathbf{seq}(M(\lambda',\mu'))$.

\smallskip

\noindent \textbf{Case $\cof(\kappa) \ne \omega$:} \ There is a unique   Tukey class in $\mathbf{seq}_\kappa(\mathcal{M})$, with type $P(\kappa) \times (\kappa,=)$.

\smallskip

\noindent \textbf{Case $\kappa=\aleph_\alpha>\aleph_0$ and $\cof(\kappa) =  \omega$:} \ The  Tukey classes in $\mathbf{seq}_\kappa(\mathcal{M})$ are parametrized by $\Delta+(\alpha+1)$, say $\mathbf{seq}(M(\lambda,\mu))$, for $(\lambda,\mu)$ in $\Delta$, and additional classes $\mathbf{seq}(N_\beta)$ for $\beta \le \alpha$. Here $\mathbf{seq}(M(\lambda,\mu)) \gte P(\kappa)\times (\lambda,=) + Q(\kappa)\times (\mu,=) + S(\kappa)$,
$\mathbf{seq}(N_\beta)\gte P(\kappa)\times (\aleph_\beta,=) + S(\kappa)$ for $\beta < \alpha$ and  $\mathbf{seq}(N_\alpha)\gte P(\kappa)\times (\kappa,=)$. Where $\alpha \ge \beta \ge \gamma$ implies $\mathbf{seq}(N_\beta) \gtq \mathbf{seq}(N_\gamma)$, $\mathbf{seq}(N_0)$
is an upper bound of all $\mathbf{seq}(M(\lambda,\mu))$, and $(\lambda,\mu) \succeq (\lambda',\mu')$ implies $\mathbf{seq}(M(\lambda,\mu)) \gtq \mathbf{seq}(M(\lambda',\mu'))$.
\end{thm}
\begin{proof}
The cases are clear from Theorem~\ref{th:seq_rep} and the natural  Tukey relations.
For the countable case, note Lemma~\ref{l:small_Pk}, which simplifies the relations.
But we must justify the `if and only if' claim in the countable case.
Specifically, we need to show: $(\lambda,\mu) \succeq (\lambda',\mu')$ if and only if $(\omom,\le_\infty) \times (\lambda,=) + (\mu,=) \gtq (\omom,\le_\infty) \times (\lambda',=) + (\mu',=)$. The reverse direction is easy by Lemma~\ref{l:sum2}.

The forward direction holds if we show $(\omom,\le_\infty) \times (\lambda,=) + (\mu,=) \not\gtq (\omom,\le_\infty) \times (\lambda',=)$ if either $\lambda<\lambda'$ or $\lambda+\mu < \lambda'+\mu'$. The second possibilty is ruled out by Lemma~\ref{l:inj}.

For a contradiction, suppose $\lambda <\lambda'$ but $( \psi_-,\phi_+)$ is a morphism witnessing $(\omom,\le_\infty) \times (\lambda,=) + (\omega,=) \gtq (\omom,\le_\infty) \times (\lambda',=)$.
Note $\phi_+$ maps $(\omom\times \lambda) \oplus \omega$ to $\omom\times \lambda'$ and $ \psi_-$  has the reverse  domain and range. For each $n$ in $\omega$, write $\phi_+(n)=(g_n,m_n)$.
Pick $f$ in $\omom$ such that, for all $n$, $f >^* g_n$.
Since $\lambda < \lambda'$, by Lemma~\ref{l:inj}, for some $m$ in $\lambda'$ and $n$ in $\omega$, $ \psi_-(f,m)=n$. But $f >^* g_n$, so it is not true that $(f,m) (\le_\infty \times =) (g_n,m_n)=\phi_+(n)$. So $( \psi_-,\phi_+)$ is not a morphism.

\medskip

For the first claim, suppose $|M|=\mu < |N|=\nu$. From Lemma~\ref{l:inj} we see $\mathbf{seq}(M) \not\gtq \mathbf{seq}(N)$. We check $\mathbf{seq}(N) \gtq \mathbf{seq}(M)$. First note (from Theorem~\ref{th:seq_rep}) that $P(\mu) \times (\mu,=) \gtq \mathbf{seq}(M)$. If $\nu$ does not have countable cofinality then $\mathbf{seq}(N)\gte P(\nu)\times (\nu,=) \gtq P(\mu) \times (\mu,=)$, and done in this case. If $\nu$ does have countable cofinality, fix $(\nu_n)_n$ as in the definition of $S(\nu)$, and pick $m$ so $\nu_m > \mu$.
Now we see $\mathbf{seq}(N) \gtq S(\nu) \gtq P(\nu_m)\times (\nu_m,=) \gtq P(\mu)\times (\mu,=)$, and done in this case, as well.
\end{proof}

Unlike the `$\kappa=\aleph_0$' and `$\kappa$ with uncountable cofinality' cases,  the authors  don't know that the specified  Tukey classes in the third case are distinct.
\begin{qu}
 Are all the relations mentioned in `Case $\kappa=\aleph_\alpha > \aleph_0$ and $\cof(\kappa)=\omega$' different under the  Tukey order?
\end{qu}
Note that the first interesting case of this question is when $\kappa=\aleph_\omega$. In this case it is consistent (when $\mathfrak{b} < \aleph_\omega$) that  the `tail', of $\mathbf{seq}_{\aleph_\omega}(\mathcal{M})$ is order isomorphic to $\omega+1$.
\begin{lem} If $\mathfrak{b} < \aleph_\omega$ and $\lambda < \lambda'$ then $P(\aleph_\omega) \times (\lambda,=) + S(\aleph_\omega) \not\gtq P(\aleph_\omega)\times (\lambda',=)$.
\end{lem}
\begin{proof} First note that, for $n \in \omega$, $\cof(P(\aleph_n) \times (\aleph_n,=)) = \mathfrak{b}.\aleph_n$.
This can be established by induction on $n$.
It follows that, $\cof(S(\aleph_\omega))  = \mathfrak{b}.\aleph_\omega$.

Now we have, $P(\aleph_\omega) \gtq ([\aleph_\omega]^\omega,ii)$ and $\cof([\aleph_\omega]^\omega,ii)\ge \aleph_{\omega+1}$, and hence $\cof(P(\aleph_\omega))\ge \aleph_{\omega+1}$.
The claimed  Tukey relation is witnessed by $\phi_+$ defined by $\phi_+((F_n)_n) = \bigcup_n F_n$ which carries cofinal subsets to cofinal subsets.

We verify this last  claim.
To this end, let $\mathcal{C}$ be a subset of $[\aleph_\omega]^{\le \omega}$ such that for all $A$ in $[\aleph_\omega]^\omega$ there is a $C$ from $\mathcal{C}$ such that $A \cap C$ is infinite. We show $|\mathcal{C}| > \aleph_\omega$.
For a contradiction, suppose $|\mathcal{C}| \le \aleph_\omega$, and enumerate it, possibly with repeats, as $\mathcal{C}=\{C_\alpha : \alpha < \aleph_\omega\}$. Let $\mathcal{C}_n = \{C_\alpha : \alpha < \aleph_n\}$, so $\mathcal{C}$ is the increasing union of the $\mathcal{C}_n$.
Note $|\mathcal{C}_n| \le \aleph_n$, so $|\bigcup \mathcal{C}_n| \le \aleph_n$, and pick $\alpha_n \in \aleph_\omega \setminus \bigcup \mathcal{C}_n$. Let $A=\{\alpha_n : n \in \omega\}$.
Now take any $C$ from $\mathcal{C}$. Then $C=C_\beta$ for a $\beta < \aleph_m$ and for some $m$. By construction, $C \cap A$ is contained in $\{\alpha_0,\ldots,\alpha_m\}$, and in particular is finite. Contradicting the cofinality of $\mathcal{C}$.

\smallskip

Now for the claim of the lemma. For a contradiction suppose $( \psi_-,\phi_+)$ is a morphism showing $P(\aleph_\omega) \times (\lambda,=) + S(\aleph_\omega) \gtq P(\aleph_\omega)\times (\lambda',=)$.
Let $C$ be cofinal in $S(\aleph_\omega)$ where $|C| \le \aleph_\omega$ (use $\mathfrak{b} < \aleph_\omega$ here).
Let $C'=\{ (F_n)_n : ((F_n)_n,\alpha)$ in $\phi_+(C)\}$. Then $|C'| < \aleph_{\omega+1}\le \cof(P(\aleph_\omega))$, so $C'$ not cofinal, and we can pick $(G_n)_n$ such that for all $((F_n)_n,\alpha)$ from $C$ we have $\neg ((G_n)_n i_\infty (F_n)_n)$.
Consider $\{ ((G_n)_n,\alpha) : \alpha < \lambda'\}$. As $\lambda < \lambda'$ by Lemma~\ref{l:inj} there is an $\alpha$ such that $ \psi_-((G_n)_n,\alpha)$ is in $S(\aleph_\omega)$.
Now, as $C$ is cofinal in $S(\aleph_\omega)$, for some $((F_n)_n,\beta)$ in $C$ we have
$ \psi_-((G_n)_n,\alpha) (i_\infty\times =) ((F_n)_n,\beta)$.
But it is not true that $((G_n)_n,\alpha) (i_\infty\times =) ((F_n)_n,\beta)$ (by choice of $(G_n)_n$). This contradicts $( \psi_-,\phi_+)$ a morphism.
\end{proof}

\begin{thm}[Counting Types] \label{th:count_seq}
If $\mathfrak{c}=\aleph_\alpha$
then there are $|\alpha|.\aleph_0$-many  Tukey types of $\mathbf{seq}(M)$, where $M$ is separable metrizable.

\medskip

In particular,

\smallskip

(0) there are at least $\aleph_0$-many and no more than $\mathfrak{c}$-many  Tukey types of $\mathbf{seq}(M)$, where $M$ is separable metrizable,

\smallskip

    (1) there are countably many  Tukey types of $\mathbf{seq}(M)$, where $M$ is separable metrizable if and only if $\mathfrak{c}< \aleph_{\omega_1}$, and

\smallskip

    (2) there are continuum, $\mathfrak{c}$,  many  Tukey types of $\mathbf{seq}(M)$, where $M$ is separable metrizable if and only if $\mathfrak{c}$ is a fixed point of the aleph function with uncountable cofinality.
    Hence the first cardinal value of $\mathfrak{c}$ so there are $\mathfrak{c}$-many types is the $\omega_1$st fixed point of the aleph function.
\end{thm}
\begin{proof} Let $\mathfrak{c}=\aleph_\alpha$.
   Consider any cardinal $\aleph_\beta \le \mathfrak{c}$. From Theorems~\ref{th:seq_rep}, \ref{th:seq_gt}  and~\ref{th:seq_real} we see that separable metrizable spaces of cardinality $\aleph_\beta$ contribute at least one and no more than $|\beta|.\aleph_0$ many $\mathbf{seq}(M)$  Tukey types.
   Thus there are $\sup \{ |\beta|.\aleph_0 : \beta \le \alpha\}=|\alpha|.\aleph_0$ types.

   Hence if $\mathfrak{c}=\aleph_\alpha$ where $\alpha < \omega_1$ then there are only countably many types, but if $\mathfrak{c} \ge \aleph_{\omega_1}$ then there are at least $\omega_1$ types.
   It follows there are countably many types precisely when $\mathfrak{c}<\aleph_{\omega_1}$.
   On the other hand, for there to be continuum many types we must have $\alpha=\aleph_\alpha$, and so $\mathfrak{c}$ is a fixed point of the aleph function. Recall also, that $\mathfrak{c}$ has uncountable cofinality.
\end{proof}
Note that it is indeed consistent that the number of types of $\mathbf{seq}(M)$ equals $\mathfrak{c}$. Via Cohen forcing we can ensure $\mathfrak{c}$ is any aleph without countable cofinality (so, for example the $\omega_1$st fixed point of the aleph function) and the forcing does not disturb the aleph fixed points.

\bibliographystyle{plain}

\bibliography{references}

\end{document}